\documentclass[english,11pt, reqno, oneside]{amsart}
\usepackage{amsmath}
\usepackage{amssymb}
\usepackage{amsthm}
\usepackage{mathtools}
\usepackage{bbm}
\usepackage{stmaryrd}
\usepackage{xcolor}
\usepackage[colorlinks=true,linkcolor=blue!85!black, citecolor = green!55!black,bookmarksdepth=3]{hyperref}
\usepackage[margin=1in, marginparwidth=2.2cm]{geometry}
\usepackage{multicol}
\usepackage{etoolbox}
\usepackage{enumitem}
\usepackage{float}
\usepackage[mathscr]{euscript}
\usepackage{graphicx}
\usepackage{subcaption}

\usepackage{microtype}

\usepackage{marginnote}

\setlist[enumerate,1]{label = (\roman*)}

\allowdisplaybreaks

\newcommand{\hair}{\ifmmode\mskip1mu\else\kern0.08em\fi}
\renewcommand{\P}{\mathbb{P}}

\newcommand{\E}{\mathbb{E}}

\newcommand{\cP}{\mathcal{P}}
\newcommand{\F}{\mathcal{F}}
\newcommand{\R}{\mathbb{R}}

\newcommand{\N}{\mathbb{N}}
\newcommand{\Z}{\mathbb{Z}}
\newcommand{\Q}{\mathbb{Q}}

\newcommand{\one}{\mathbbm{1}}
\newcommand{\intint}[1]{\llbracket #1 \rrbracket}
\newcommand{\mc}{\mathcal}
\newcommand{\mf}{\mathfrak}

\newcommand{\srt}{\shortrightarrow}

\DeclareMathOperator{\diam}{diam}

\renewcommand{\S}{\mathcal S}
\newcommand{\A}{\mathcal A}
\newcommand{\h}{\mathfrak h}
\renewcommand{\epsilon}{\varepsilon}

\renewcommand{\emptyset}{\varnothing}

\newcommand{\rec}{\mathrm{rec}}

\newcommand{\wdf}{\mathcal{D}}
\newcommand{\NC}{\mathsf{NC}}

\newcommand{\cL}{\mathcal{L}}

\newcommand{\PT}{\mathsf{PT}}

\makeatletter
\newcommand*\bigcdot{\mathpalette\bigcdot@{.45}}
\newcommand*\bigcdot@[2]{\mathbin{\vcenter{\hbox{\scalebox{#2}{$\m@th#1\bullet$}}}}}
\makeatother

\newtheorem{maintheorem}{Theorem}
\newtheorem{maincorollary}[maintheorem]{Corollary}

\newtheorem{theorem}{Theorem}[section]
\newtheorem*{theorem*}{Theorem}
\newtheorem*{proposition*}{Proposition}
\newtheorem{proposition}[theorem]{Proposition}
\newtheorem*{corollary*}{Corollary}
\newtheorem{corollary}[theorem]{Corollary}
\newtheorem{lemma}[theorem]{Lemma}

\theoremstyle{definition}
\newtheorem{definition}[theorem]{Definition}

\newtheorem*{notation}{Notation}
\newtheorem{remark}[theorem]{Remark}

\title[Comparisons of the Airy difference profile to Brownian local time]{Local and global comparisons of the  Airy difference profile to Brownian local time}

\author{Shirshendu Ganguly}
\address{Shirshendu Ganguly, Department of Statistics, U.C. Berkeley, Berkeley, CA, USA}
\email{sganguly@berkeley.edu}

\author{Milind Hegde}
\address{Milind Hegde, Department of Mathematics, U.C. Berkeley, Berkeley, CA, USA}
\email{milind.hegde@berkeley.edu}

\begin{document}

\begin{abstract}
There has recently been much activity within the Kardar-Parisi-Zhang universality class spurred by the construction of the canonical limiting object, the parabolic Airy sheet $\S:\R^2\to\R$ \cite{dauvergne2018directed}. The parabolic Airy sheet provides a coupling of parabolic Airy$_2$ processes---a universal limiting geodesic weight profile in planar last passage percolation models---and a natural goal is to understand this coupling. Geodesic geometry suggests that the difference of two parabolic Airy$_2$ processes, i.e., a difference profile, encodes important structural information. This difference profile $\wdf$, given by $\R\to\R:x\mapsto \S(1,x)-\S(-1,x)$, was first studied by Basu, Ganguly, and Hammond \cite{basu2019fractal}, who showed that it is monotone and almost everywhere constant, with its points of non-constancy forming a set of Hausdorff dimension $1/2$. 
Noticing that this is also the Hausdorff dimension of the zero set of Brownian motion, we adopt a different approach. Establishing previously inaccessible fractal structure of $\wdf$, we prove, on a global scale, that $\wdf$ is absolutely continuous on compact sets to Brownian local time (of rate four) in the sense of increments, which also yields the main result of \cite{basu2019fractal} as a simple corollary. Further, on a local scale, we explicitly obtain Brownian local time of rate four as a local limit of $\wdf$ at a point of increase, picked by a number of methods, including at a typical point sampled according to the distribution function $\wdf$. Our arguments rely on the representation of $\S$ in terms of a last passage problem through the parabolic Airy line ensemble and an understanding of geodesic geometry at deterministic and random times.
\end{abstract}

\maketitle

\setcounter{tocdepth}{1}
\tableofcontents

\section{Introduction and main result}
\label{s.intro}

The Kardar-Parisi-Zhang (KPZ) universality class refers to a broad family of models of one-dimensional random growth which are believed to exhibit certain common features, such as universal scaling exponents and limiting distributions. A plethora of models are believed to lie in this class, including asymmetric exclusion processes, first passage percolation, directed polymers in random environment, and the stochastic PDE known as the KPZ equation. Nonetheless, in spite of the breadth of models thought to lie in the class, only a handful have been rigorously shown to do so.

An important subclass of models believed to be members of the KPZ universality class is known as last passage percolation (LPP). While the microscopic details depend on the specification, in general terms they all consist of an environment of random noise through which \emph{directed} paths travel, accruing the integral of the noise along it---a quantity known as energy or weight. Typically, two endpoints in the environment are fixed, and a maximization is done over the weights of all paths with these endpoints. A path which achieves the \emph{maximum} weight is called a \emph{geodesic}, although, traditionally, the latter is used to denote \emph{shortest} paths in a metric space.

To facilitate our discussion here without getting into technical definitions, we imagine a limiting last passage percolation model defined on $\R\times[0,1]$, i.e., an infinite strip with height one; the height can be thought of as a \emph{time} parameter. (The purpose of this imagined model is to explain the interpretations of certain limiting objects we will introduce, and is not meant to be one which is rigorously defined.) The precise distribution of the noise is not important for our expository purposes, but it can be thought of as translationally invariant and independent. We consider paths $\gamma:[0,1]\to\R$, where $\gamma(t)$ denotes the position of the path at time $t$; the directedness constraint on paths is implemented by the requirement that they are functions, and so cannot ``backtrack'' and have multiple values at any height, i.e., time. The weight of a given path can be thought of as the integral of the noise along the path in some sense. The last passage value from $(y,0)$ to $(x,1)$ is the maximum weight over all paths $\gamma$ with $\gamma(0) = y$ and $\gamma(1) = x$.

A canonical limiting object in the KPZ universality class is known as the parabolic Airy$_2$ process, $\cP_1:\R\to\R$ \cite{prahofer2002PNG}. (The subscript of $1$ is due to $\cP_1$ being the first in a family of random curves, the parabolic Airy line ensemble, which plays a central role in this paper.) In terms of our last passage percolation model, $\cP_1(x)$ should be thought of as encoding the weight of the geodesic from $(0,0)$ to $(x,1)$. {The term ``parabolic'' is included in the name because $\cP_1$ is obtained by subtracting the parabola $x^2$ from the Airy$_2$ process, which is stationary.}

Note that while the endpoint $x$ is allowed to adopt any real value, the starting point is fixed to be zero in the interpretation of $\cP_1(x)$. In our LPP model, of course, any starting point is also allowed (i.e., $\gamma(0)$ can take any real value), but the joint distributions of geodesic weights with differing starting and ending points is not captured in the one-dimensional object $\cP_1$. 

There is a richer universal object that encodes this larger class of joint distributions, providing a joint coupling of the geodesic weights as the endpoints are allowed to vary arbitrarily, known as the parabolic Airy sheet $\S:\R^2\to\R$. While first conjectured to exist in \cite{corwin2015renormalization}, this was only recently proved in the important work \cite{dauvergne2018directed} (with assistance from \cite{dauvergne2018basic}). The value of the parabolic Airy sheet $\S(y,x)$ should be thought of as the weight of the geodesic from $(y,0)$ to $(x,1)$.

The parabolic Airy sheet admits invariance under appropriate scalings guided by certain critical exponents, much like Brownian motion. Hence, as in the case of the latter, $\S$ is expected to exhibit various fractal or self-similar properties.  

The inquiry into such aspects, including the existence of random fractal sets and their fractal dimensions, is indeed a well-established theme in probability theory and statistical mechanics. Particularly important examples include: (i)  the zero set of Brownian motion (which has an important connection to our methods that we will return to), (ii) the set of exceptional times in dynamical critical planar percolation where an infinite cluster is present, which is proven to have Hausdorff dimension $\frac{31}{36}$ \cite{garban2010fourier} on the honeycomb lattice (and conjectured to have the same on the Euclidean lattice), and (iii) the study of Schramm-Loewner evolutions in connection to scaling limits of interfaces at criticality in various statistical mechanics models; here too Hausdorff dimensions are known, in this case of the curves themselves \cite{rohde2005basic,beffara2008dimension}.

The above has led to the study of random fractal geometry within KPZ, which is still at a rather nascent stage, notwithstanding some important recent advances. A natural object to investigate to unearth fractal geometry in this setting is  the coupling, as defined by $\S$, of weights from different starting points.
The first such progress, which is closely related to our work, in this direction was made in  \cite{basu2019fractal}. 
To make things precise, let us first define the object of study.

Fix $y_a,y_b\in\R$ with $y_b > y_a$. We consider the random function $\wdf:\R\to\R$ given by 
\begin{equation}\label{e.W defn}
\wdf(x) = \S(y_b, x) - \S(y_a, x).
\end{equation}
We will call $\wdf$ a \emph{weight difference profile}; it is the difference in the weight of two geodesics with differing but fixed starting points and a  common ending point as the latter varies. A first fact about $\wdf$ sets the stage for our work, and indicates why it encodes useful information about the coupling between the weights from two sources.

\begin{lemma}\label{l.W increasing}
$\wdf$ is a continuous non-decreasing function almost surely.
\end{lemma}

This has been proved several times in the literature, for example \cite[Lemma 9.1]{dauvergne2018directed} or \cite[Theorem 1.1(2)]{basu2019fractal}, but we will include the simple proof ahead in Section~\ref{s.geometry} for completeness. 
Lemma~\ref{l.W increasing} implies that at any given point $x$, $\wdf$ is either constant in a neighbourhood of $x$ or increasing at $x$ on at least one side. The latter is defined precisely as follows.

\begin{definition}\label{d.pt of increase}
A point $x\in\R$ is a \emph{non-constant} point of a continuous function $f:\R\to\R$ if there does not exist any $\epsilon>0$ such that $f(y) = f(x)$ for all $y\in (x-\epsilon ,x + \varepsilon)$. The set of non-constant points of $f$ is denoted by $\NC(f)$.
\end{definition}


It turns out that $\wdf$ is almost everywhere constant, with probability one. Further, it is an easy argument that the non-constant points of a continuous non-decreasing function must form a perfect set, i.e., be closed and have no isolated points. Perfect sets must necessarily be uncountable. In particular, $\wdf$ can be interpreted as the distribution function of a random measure supported on the uncountable set $\NC(\wdf)$. Canonical examples of a similar nature include the Cantor function or Brownian local time. Associated to any such set is its fractal dimension, which quantifies how ``sparse" the set is. With the dimension of the mentioned two examples being classically known, one is led to ask: what is the fractal dimension of $\NC(\wdf)$?

This question was originally raised and answered in \cite{basu2019fractal}, where it was shown that the Hausdorff dimension of $\NC(\wdf)$ is $\frac{1}{2}$ almost surely. (The definition of the Hausdorff dimension of a set is recalled ahead in Definition~\ref{d.hd} for the reader's convenience.)


It is well-known that the set of non-constant points of Brownian local time (to be denoted by $\cL$ henceforth), i.e, the set of zeroes of Brownian motion, also has Hausdorff dimension one-half. Inspired by a question of Manjunath Krishnapur about a possible connection between the latter and ${\mathsf{NC}}(\wdf)$, the present article adopts a different method to study $\wdf$. The main theorems establish comparison results between $\cL$ and $\wdf$, and further, as a  straightforward corollary of our methods, the main result of \cite{basu2019fractal} about the Hausdorff dimension of $\NC(\wdf)$ is quickly obtained (see Corollary~\ref{c.HD is one-half}).

%
Now, one can enquire about a comparison to $\cL$ globally as well as locally. We obtain results for both scales and start by discussing the former.
%

\subsection{A Brownian local time comparison on a global scale}\label{s.intro global scale}

Comparisons of objects arising in KPZ to Brownian counterparts are by now a well-established theme. One form of comparison typically involves absolute continuity of the relevant objects to each other. For example, for the parabolic Airy$_2$ process $\cP_1$, \cite{corwin2014brownian} proved that $x\mapsto \cP_1(x)-\cP_1(a)$ is absolutely continuous, as a process on a compact interval $[a,b]$, to Brownian motion; this was strengthened by getting bounds on a superpolynomial moment of the resulting Radon-Nikodym derivative in \cite{hammond2016brownian} (though with an affinely shifted version of $\cP_1$ being compared to Brownian bridge) and \cite{calvert2019brownian} (comparing the originally introduced increment to Brownian motion).

As we mentioned earlier, the process $\cP_1$ should be thought of as encoding the weight of geodesics with starting point fixed at $(0,0)$. One can also consider other initial conditions, which are parametrized by functions $\h_0:\R\to\R\cup\{-\infty\}$. The interpretation of $\h_0$ is that paths starting at $(y,0)$ are given an extra (though possibly negative) weight $\h_0(y)$ in addition to the usual weight they obtain on their journey through the environment. Maximizing this augmented weight over all paths with given endpoint $(x,1)$ (with unconstrained starting point) gives the \emph{KPZ fixed point} $\h_1$, which was first constructed in the breakthrough paper \cite{matetski2016kpz}. 

(Here the subscript of 1 in $\h_1$ indicates that we are at time 1, and we can analogously defined $\h_t$ by making the purely notational modifications in our imagined continuum LPP model to allow all times in $(0,\infty)$. We will discuss fractal properties of the process $t\mapsto \h_t$ ahead in our review of previous literature.)

Hammond used his results in \cite{hammond2016brownian} along with a detailed study of geodesic geometry in a model known as Brownian LPP to obtain a comparison of $\h_1$ with Brownian motion on a unit order, i.e., global, scale for quite a general class of initial data in \cite{hammond2017patchwork} (and aided by two further works, \cite{hammond2017modulus,brownianLPPtransversal}). These works all relied on the Brownian Gibbs property, an invariance under resampling enjoyed by the parabolic Airy line ensemble, which shall play a central role in this paper as well. We will introduce it slightly later in Section~\ref{s.intro.proof strategy}, and formally in Definition~\ref{d.bg}.

\cite{hammond2017patchwork}'s result on a unit-order comparison of $\h_1$ to Brownian motion roughly stated that any given compact interval $[a,b]$ can be decomposed into a random number of random subintervals (``\emph{patches}''), and that the restriction of $\h_1$ to each patch (a ``\emph{fabric function}'') enjoyed a comparison to Brownian bridge (later upgraded to Brownian motion in \cite{calvert2019brownian}) with a certain number of moments of the Radon-Nikodym derivative finite.
 More recently, \cite{sarkar2020brownian} obtained a comparison with a \emph{single} patch, but with no further regularity information than absolute continuity. Our main result is comparable to \cite{sarkar2020brownian} and proves an absolute continuity comparison without need for patches. We will say more about \cite{sarkar2020brownian} and their proof approach later in Sections~\ref{s.intro.proof strategy} and \ref{s.geometry}.

Let $\cL:[0,\infty)\to[0,\infty)$ be the local time at zero associated to Brownian motion of rate $\sigma^2$ begun at the origin (the definition of local time is recalled in Definition~\ref{d.local time}). We will always explicitly mention the relevant diffusivity $\sigma$, and so we will drop the dependence on the same in the notation~$\cL.$

We now come to the first main result of this paper.

\begin{maintheorem}\label{t.patchwork}
Let $[c,d] \subset \R$. Then $\wdf(\cdot) - \wdf(c)|_{[c,d]}$ is absolutely continuous to $\cL(\cdot) - \cL(1)|_{[1,1+d-c]}$, where $\cL$ is a Brownian local time of rate four.
\end{maintheorem}



We only consider an increment of $\wdf$ because it may take negative values, unlike $\cL$. Note also that the comparison is with an increment of $\cL$ with initial point $1$ and not $0$; this is because $\wdf$ will almost surely be constant in a neighbourhood of $c$, while $\cL$ almost surely increases at the origin while remaining constant around $1$.

The source of the rate of four in Theorem~\ref{t.patchwork} is that $\S$ (recall the definition from the discussion preceding \eqref{e.W defn}), roughly speaking, locally has diffusion rate two. This is simply the normalization adopted in the definition of the Airy$_2$ process to ensure that $\cP_1$, or equivalently $\S(0,\cdot)$, is obtained from the latter by simply subtracting off the parabola $x^2$ without any multiplicative factor. This local diffusivity rate is reflected in the fact that the distributional limit of $\varepsilon^{-1/2}(\S(0, x+\varepsilon t) - \S(0,x))$ is two-sided Brownian motion of rate two \cite{hagg2008local,quastel2013local,cator2015local}. The aforementioned Brownian Gibbs property of the yet-to-be-introduced parabolic Airy line ensemble also makes this evident (see Section~\ref{s.geometry} for further elaboration). 
Our proof will express $\wdf$ as the running maximum of a difference of two processes which enjoy a joint comparison to independent Brownian motions of rate two. The difference of these processes is like a rate four Brownian motion. We then use L\'evy's identity to relate the running maximum to $\cL.$

\begin{remark}
An earlier version of this paper posted to the arXiv only proved a comparison of $\wdf$ to $\cL$ on random patches (see Version 1 on arXiv for a precise statement). The proof presented here of the stronger full absolute continuity result was subsequently obtained by a slight modification of our previous arguments.  In the meantime, Dauvergne has also obtained a result quite similar to Theorem~\ref{t.patchwork} using related methods in \cite{dauvergne2021isometries}.
\end{remark}

Given Theorem~\ref{t.patchwork}, it is easy to identify the Hausdorff dimension of $\NC(\wdf)$ using that the support of $\cL$ almost surely has Hausdorff dimension one-half, along with the countable stability property of Hausdorff dimension.

\begin{maincorollary}[Theorem~1.1 of \cite{basu2019fractal}]\label{c.HD is one-half}
The Hausdorff dimension of $\NC(\wdf)$ is equal to one-half almost surely.
\end{maincorollary}

However, to maintain the flow of exposition and properly recall these properties, we defer the proof of Corollary~\ref{c.HD is one-half} to Section~\ref{s.proof of HD theorem}. 


\subsection{Brownian local time in the local limit}

In contrast to the global scale result, our main result on the local scale explicitly obtains $\cL$ as a local limit. Before stating it precisely, we make a few remarks. 

We will be considering limits of the form $\varepsilon^{-1/2}(\wdf(\tau+\varepsilon t) - \wdf(\tau))$ where $\tau$ is a random time. First, note that $\tau$ indeed needs to be random: at a deterministic time, $\wdf$ is almost surely constant, and the local limit will be trivial. Further, for the same reason, $\tau$ needs to be almost surely a point of increase of $\wdf$ for there to be any hope of obtaining $\cL$ in the limit.

Now, there are a number of ways of choosing a point of increase of $\wdf$. For instance, we may fix $\lambda\in\R$ and consider the first point of increase $\tau_\lambda$ following $\lambda$. This would, in some sense, be a choice size-biased by the length of the flat portion preceding $\tau_\lambda$; also, it excludes any of the ``interior'' points of increase of $\wdf$ from being considered. Alternately, we may consider the first time $\rho^h$ that $\wdf$ hits a given level $h\in\R$, or the first time $\rho_c^h$ that $\wdf$ hits $\wdf(c) + h$ for given $h> 0$ and a fixed $c$; intuitively (and as we will prove), these are points of increase of $\wdf$ and will typically be ``interior'' ones. A fourth method of choice, which addresses the size-biasing issue, could be to choose, for an interval $[c,d]$, a point $\smash{\xi_{[c,d]}}$ \emph{uniformly} from all the non-constant points of $\wdf$ on $[c,d]$, by sampling from the probability measure on $[c,d]$ with distribution function $(\wdf(\cdot) - \wdf(c))/(\wdf(d)-\wdf(c))$, conditionally on the event that this is a non-zero measure, which we will show has positive probability in Lemma~\ref{r.positive prob of non-empty}. (Indeed, distributions at random times sampled according to such local times have been studied, for example, in  the context of dynamical critical percolation \cite{hammond2015local}.)

We prove that the local limit at any of these four types of random times is Brownian local time:

\begin{maintheorem}\label{t.local limit}
Let $\tau$ be equal to either $\tau_\lambda$, $\rho^h$, $\rho_c^h$, or $\xi_{[c,d]}$ (the last conditionally on $\NC(\wdf)\cap[c,d]\neq\emptyset$) as above.  Then,
$$\lim_{\varepsilon\to 0} \varepsilon^{-1/2}\bigl(\wdf(\tau + \varepsilon t) - \wdf(\tau)\bigr) = \cL(t),$$
where, as before, $\cL:[0,\infty)\to [0,\infty)$ is the local time at the origin of rate four Brownian motion, and the limit is in distribution in the topology of uniform convergence on compact sets of $[0,\infty)$.
\end{maintheorem}

For completeness, we will also prove that each of the mentioned random points are almost surely points of increase of $\wdf$, and, as mentioned, that $\NC(\wdf)\cap[c,d] \neq \emptyset$ with positive probability.

Several recent papers have proved an analogous statement that the local limit of the KPZ fixed point at a fixed location is two-sided Brownian motion.
In the case of the parabolic Airy$_2$ process, this was first done in various senses in \cite{hagg2008local,cator2015local,quastel2013local}. Under general initial conditions, this is proven in the sense of convergence of finite dimensional distributions in \cite[Theorem~4.14]{matetski2016kpz}. In the stronger topology of uniform convergence on compact sets, the convergence is implied by combining the statement of absolute continuity to Brownian motion on compact intervals (\cite[Theorem~1.2]{sarkar2020brownian}) with a statement that a local limit of a process which is absolutely continuous to Brownian motion is itself Brownian motion (see \cite[Lemma~4.3]{dauvergne2020three}, also cited here as Lemma~\ref{l.local limit BM}).



We mention another interesting recent work investigating a local limit, though in a slightly different sense. \cite{dauvergne2020three} considers the local limit of the environment around the geodesic, in the continuum LPP model where the environment is given by the \emph{directed landscape}, constructed in \cite{dauvergne2018directed}, which is a richer scaling limit than $\S$ and, in terms of our earlier imagined continuum LPP model, encodes the joint distribution of LPP weights when even the starting and ending heights, fixed to be $0$ and $1$ for $\S$, are allowed to adopt any values $s<t$. \cite{dauvergne2020three} shows that the local limit of the environment around an interior point of the geodesic in this random environment as well as the local limits of the geodesic and its weight can be described in terms of an LPP problem driven by $\S$ and boundary data involving classical objects such as two-sided Brownian motion and the Bessel-3 process.

(It is this continuum model of LPP defined by the directed landscape which is a rigorously realized version of the continuum model we have been using for expository purposes, in the sense that limiting objects such as $\cP_1$ and $\h_t$ are as described above with the directed landscape LPP model. However, the environment defined by the directed landscape does not consist of independent noise. A continuum LPP model with some sort of independent noise whose weights are encoded by the directed landscape, i.e., a zero temperature analogue to the continuum directed random polymer \cite{alberts2014continuum} whose free energy is encoded by the KPZ equation, has not yet been constructed.)

\subsection{Prior work on fractal aspects of KPZ}
As has been mentioned, \cite{basu2019fractal} initiated the study of fractal geometry within KPZ by identifying the Hausdorff dimension of the set $\NC:=\NC(\wdf)$. At a high level,  the argument for the upper bound on the Hausdorff dimension relied on showing that $\NC$ is a subset of the set of points $x$ admitting geodesics from $(y_a, 0)$ and $(y_b,0)$ to $(x,1)$ that are disjoint except for the common endpoint. We will elaborate on this slightly when we contrast our proof strategy with that of \cite{basu2019fractal} in Section~\ref{s.comparison to bgh}.

An exact form of this correspondence between geodesic disjointness and the non-constant points of $\wdf$, i.e., that the sets are equal, was proved shortly afterwards in \cite{bates2019hausdorff}. Additionally, \cite{bates2019hausdorff} identified the Hausdorff dimension of the set of points $(y,x)$, as a subset of $\R^2$, such that there are at least two disjoint geodesics (except for the common start and endpoints) from $(y,0)$ to $(x,1)$. This set's dimension was also shown to be one-half. In both these results of \cite{bates2019hausdorff} the geodesics are defined in terms of the directed landscape, similar to the earlier mentioned \cite{dauvergne2020three}. 

There are two further recent studies of fractal dimension within KPZ, \cite{KPZfixedptHD} and \cite{das2021law}. 

In the first \cite{KPZfixedptHD}, instead of studying the parabolic Airy sheet or directed landscape directly, the KPZ fixed point is studied. Recall its definition as a process in $t$ from Section~\ref{s.intro global scale}.
\cite{KPZfixedptHD} identifies as $\frac{2}{3}$ the Hausdorff dimension of the set of exceptional times $t>0$ where $\h_t$ has multiple maximizers, for a broad class of initial data $\h_0$, conditionally on the set of exceptional times being non-empty (an event which is shown to have positive probability, and conjectured to be an almost sure event). In the geodesic picture, these exceptional times are exactly when there is not a unique geodesic with initial condition given by $\h_0$ and unconstrained ending point.

The second \cite{das2021law} investigates the upper and lower laws of iterated logarithm (LIL) in time, at a fixed spatial location, for the solution to the KPZ equation (a canonical stochastic PDE in the KPZ universality class) started from narrow-wedge initial condition. They show that the upper LIL occurs at scale $\smash{(\log\log t)^{2/3}}$ and the lower at scale $\smash{(\log\log t)^{1/3}}$. This agrees with the scales of laws of iterated logarithm proved previously in prelimiting LPP models in \cite{ledoux2018law,basu2019lower}. The result of \cite{das2021law} of relevance from the point of view of fractal geometry is their further study of the level sets, parametrized by $\alpha$, i.e.,  times $t$ where the solution exceeds $\alpha(\log\log t)^{2/3}$ and their Hausdorff dimensions. They also establish an interesting transition from mono-fractal to multi-fractal behaviour under an appropriate exponential time change.  

{Finally, in a recent preprint \cite{sep2021busemann}, the authors announce that in forthcoming work they will prove that, in the model of Brownian LPP, the set of points $(m,t)$ in the semi-discrete plane for which there is a random direction $\theta$ such that there exist two semi-infinite geodesics (which will be disjoint) in direction $\theta$ starting from $(m,t)$ has Hausdorff dimension $\frac{1}{2}$.}

\subsection{Remarks on proof strategy and organization of the paper}
\label{s.intro.proof strategy}
In this section we sketch the basic ideas behind Theorems~\ref{t.patchwork} and \ref{t.local limit}. Corollary~\ref{c.HD is one-half} is a straightforward consequence of Theorem~\ref{t.patchwork} and the fact that the support of $\cL$ almost surely has Hausdorff dimension one-half. The approach we take relies on what may be called a continuous Robinson-Schensted-Knuth correspondence which \cite{dauvergne2018directed} used to define the parabolic Airy sheet $\S$.


\subsubsection{The parabolic Airy sheet via continuous RSK}
\cite{dauvergne2018directed} defined the parabolic Airy sheet $\S$ in terms of a limiting semi-discrete LPP problem in an environment defined by the parabolic Airy line ensemble $\cP$. The latter is an infinite $\N$-indexed collection of random non-intersecting curves whose lowest indexed, but highest in value, curve is $\cP_1$, the parabolically shifted Airy$_2$ process; see Figure~\ref{f.airy line ensemble}. 

\begin{figure}[h]
\centering
\includegraphics[width=0.35\textwidth]{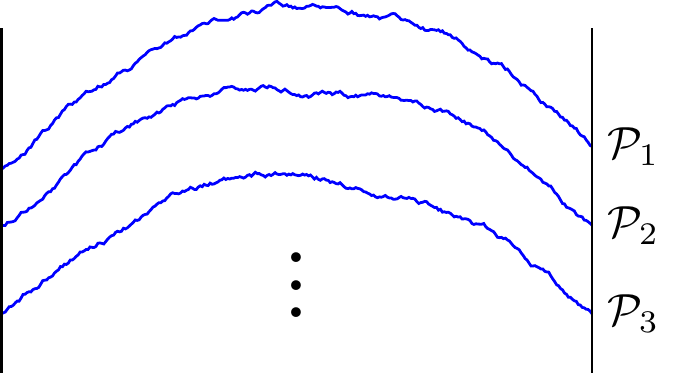}
\caption{A depiction of the parabolic Airy line ensemble.}\label{f.airy line ensemble}
\end{figure}

The above mentioned last passage problem in $\cP$  can be described by fixing a starting coordinate on the $k$\textsuperscript{th} line (for some $k\ge 1$) and an ending coordinate on the first (i.e., top) line, and considering up-right paths between the two. The weight of a given path is given by the sum (over $i$) of increments of values of $\cP_i$ along the interval that the path spends on the $i$\textsuperscript{th} line. If the starting point is $(y,k)$ and ending point is $(x,1)$, we denote this by $\cP[(y,k)\to(x,1)]$.

It was shown in \cite{dauvergne2018directed} that the set of last passage values obtained this way encodes $\S$. At a high level, the increments of $\S$ with a fixed starting point, i.e., of the form $\S(y,x) - \S(y,z)$, are defined as the limit of a difference $\cP[((y)_k, k)\to(x,1)] - \cP[((y)_k, k)\to(z,1)$ of LPP values in $\cP$. Here $\{(y)_k\}_{k\in\N}$ is a sequence of deterministic points, depending on $y$, which is such that $(y)_k\to-\infty$ in a parabolically curved manner as $k\to\infty$; see Definition~\ref{d.airy sheet}. Similar formulations have recently been used to construct the extended directed landscape in \cite{dauvergne2021disjoint}.

Note that here the increment of $\S$ is with shared starting point and differing ending points, while our object of interest $\wdf(x) = \S(y_b,x) - \S(y_a,x)$ has differing starting point but shared ending point. Thus the description of the former type of increment of $\S$ in terms of a difference of LPP values in $\cP$ is not directly useful for us; we need a LPP description that holds for $\S(y,x)$ itself. This is rather difficult and in fact an open problem (\cite[Conjecture~14.2]{dauvergne2018directed}). 

This issue was encountered by \cite{sarkar2020brownian} as well, who bypassed it by considering an LPP problem in $\cP$  with appropriate \emph{boundary data} which encodes $\S(y,x).$

\subsubsection{LPP with boundary data}\label{s.intro.boundary data}
\cite[Lemma 3.10]{sarkar2020brownian} (and cited here as Lemma~\ref{l.Airy sheet to boundary LPP problem}) says that, for given $\lambda\in\R$, there exist random numbers $\{b^\lambda_i\}_{i\in\N}$ such that, for $x\geq 0$,
\begin{equation}\label{e.boundary problem first mention}
\begin{split}
\S(y_b,\lambda+x) &= \sup_{i\in\N} \left\{b^\lambda_i + \cP[(\lambda, i) \to (\lambda+x,1)]\right\}.
\end{split}
\end{equation}
Thus the RHS is the last passage value in $\cP$ from the set $\{\lambda\}\times \N$ to $(\lambda+x,1)$ with boundary values $\{b^\lambda_i\}_{i\in\N}$, and is what we will refer to as an LPP problem with boundary data.
In this case, the boundary data $\smash{\{b^\lambda_i\}_{i\in\N}}$ implicitly encodes the weight contribution coming from the initial segments (i.e., up to $\lambda$) of the geodesics between the points $((y)_k, k)$ and $(\lambda+x,1)$ as $k\to \infty$ when $y=y_b$.

In fact, \cite{sarkar2020brownian} interprets $b^{\lambda}_i$ as the weight of an \emph{infinite} geodesic formed from the geodesics from $((y)_k, k)$ to $(x,i)$ as $k\to\infty$, where the meaning of the infinite geodesic's weight must be reformulated with proper centering since $\cP[((y)_k, k) \to (x,i)]$ will diverge as $k\to\infty$. While we will make use of this infinite geodesic as intuition in our discussion of our proof strategy in this section, and discuss it in more detail in Section~\ref{s.geometry}, the actual arguments will not rely on this except for the use of one statement from \cite{sarkar2020brownian}, recorded as Lemma~\ref{l.Airy sheet to boundary LPP problem}.

An important point of convenience is that we do not need to consider a representation like \eqref{e.boundary problem first mention} for $\S(y_a,x)$. Indeed, $\S$ enjoys the stationarity property that $\S(\cdot + t, \cdot+t) = \S(\cdot, \cdot)$ for any fixed $t\in\R$ as a process on $\R^2$. This allows us to assume without loss of generality that $y_a = 0$. This is a key choice because it is proved in \cite{dauvergne2018directed} that the coupling between $\S$ and $\cP$ also has the property that $\S(0, x) = \cP_1(x)$ for all $x\in\R$: we obtain a direct representation of $\S(y_a, x)$ in terms of $\cP$ for free. For our arguments, it would thus be particularly convenient and cause no loss in generality to assume $y_a = 0$.

We will need to know that expressions such as the right hand side of \eqref{e.boundary problem first mention} enjoy absolutely continuity properties with respect to Brownian motion. Before we outline in a few subsections ahead how to make such assertions, let us mention the important fact about $\cP$ that is the source of all Brownian comparisons we make, namely that it enjoys the \emph{Brownian Gibbs} property \cite{corwin2014brownian}. This means that conditional on everything apart from the top $k$ curves of $\cP$ on an interval $[a,b]$, the distribution of $(\cP_1, \ldots, \cP_k)$ on $[a,b]$ is given by $k$ independent rate two Brownian bridges between the correct endpoints which are conditioned to not intersect {each other or the lower curve $\cP_{k+1}$}. An immediate implication is that the joint law of $\cP_i(\cdot) - \cP_i(a)$, for $i=1, \ldots, k$, is absolutely continuous with respect to the law of $k$ independent rate two Brownian motions on $[a,b]$. The formal definition of the Brownian Gibbs property is given in Definition~\ref{d.bg}.


Before proceeding further it would be convenient to setup some notation.  Given $\cP$,
for $i<j$ and $\lambda \in \R$, let
\begin{equation}\label{comnot}
\cP^\lambda_{j \shortrightarrow i}(x) : = \cP[(\lambda,j) \to (\lambda+x, i)].
\end{equation}
For our applications we will often be using the case that $i=1$. Note also that when $i=j$, $\cP^\lambda_{i\srt i}(x)=\cP_i(\lambda+x)-\cP_i(\lambda)$ measures the increment of the $i$\textsuperscript{th} line $\cP_i$ across the interval $[\lambda, \lambda+x]${; for ease of notation, we will often denote $\smash{\cP^\lambda_{i\srt i}}$ by simply $\smash{\cP^\lambda_i}$.}


\subsubsection{A representation for $\wdf$ in a special case.} \label{ptdiscuss}
To explain how we will relate $\wdf$ to Brownian local time, let us first consider a special case. 
%
%
We consider the case that the maximum in \eqref{e.boundary problem first mention} is achieved at $j=2$. Then for $\lambda\in\R$ fixed and all $x\geq 0$, 
\begin{equation}\label{e.wdf special case representation}
\wdf(\lambda+x) = \S(y_b, \lambda+x) - \S(y_a, \lambda+x) = b^\lambda_2 + \cP^\lambda_{2\srt 1}(x) -\cP_1(\lambda+x).
\end{equation}
Observe that, by our choice of $y_a=0$, we have $\S(y_a,\lambda+x) = \cP_1(\lambda+x)$. In particular, there is \emph{no boundary data corresponding to $y_a$}. 
%

Now, it is easy to see that
\begin{equation}\label{e.2-LPP}
\cP^\lambda_{2\shortrightarrow 1}(x) = \cP[(\lambda,2) \to (\lambda+x, 1)] = \cP^\lambda_{1}(x)  + \max_{0\leq s\leq x}\left(\cP^\lambda_{2}(s) - \cP^\lambda_{1}(s)\right).
\end{equation}
%
The important feature of this formula is that it involves the \emph{maximum} of the difference of $\cP_2$ and $\cP_1$. Applying \eqref{e.2-LPP} to \eqref{e.wdf special case representation}, recalling that $\cP_1^\lambda(x) = \cP_1(\lambda+x) - \cP_1(\lambda)$, and observing from \eqref{e.wdf special case representation} that $\wdf(\lambda) = b_2^\lambda - \cP_1(\lambda)$, we see that
\begin{equation}\label{e.easy two case}
\wdf(\lambda+x) - \wdf(\lambda) = \max_{0\leq s\leq x}\left(\cP^\lambda_{2}(s) - \cP^\lambda_{1}(s)\right).
\end{equation}
By the Brownian Gibbs property, $(\cP^\lambda_{1}, \cP^\lambda_{2})$ is absolutely continuous to a pair of independent Brownian motions of rate two, and so \eqref{e.easy two case} is absolutely continuous, \emph{on the whole interval} $x\in[0,T]$ for any fixed $T$, to the running maximum of rate four Brownian motion. By L\'evy's famous identity, the latter has exactly the distribution of the local time at the origin of rate four Brownian motion (see Proposition~\ref{p.levy identity} to recall this).


Of course, the above nice representation only holds when the maximizing index in \eqref{e.boundary problem first mention} is $j=2$, which can be thought of as the case where the infinite geodesic to $(\lambda+x,1)$ is at line $2$ at $\lambda$. In the general case, the geodesic may be at any line at $\lambda$. The central idea to handle the general case is the following. For the geodesic to reach $(\lambda+x,1)$, it must enter the second line at some point. If we can encode the weight accumulated by the geodesic up until this point in a new process, then $\wdf(\lambda+x)-\wdf(\lambda)$ will have a representation analogous to \eqref{e.easy two case} with this new process in place of $\cP^\lambda_2$.

Essentially, we are looking for a \emph{boundary data process} up till the second line, which we also need to be Brownian-like in a suitable sense. As this process plays an important role in our arguments, we turn to describing it next.

\subsubsection{A boundary data process}



To define a boundary data process up to the $i$\textsuperscript{th} line, it is not possible to simply use the definition of $b^\lambda_i$ as a process in $\lambda$. This is because the latter is defined in \cite{sarkar2020brownian} for only a fixed $\lambda\in\R$, but we will need to know properties such as the continuity of the process in $\lambda$, which is not obvious from the definition  of $b^{\lambda}_i$ even if it could be extended from its definition for fixed $\lambda$. To emphasize this distinction, we call the boundary data process $\smash{Z^{\lambda, b}_i}$, where $\smash{Z^{\lambda, b}_i(x)}$ should be thought of as $\smash{b^{\lambda+x}_i}$.

To define $Z^{\lambda, b}_i(x)$, we consider the LPP problem from the vertical line at $\lambda$, with boundary data $\smash{\{b^\lambda_j\}_{j\geq i}}$, to $(\lambda+x, i)$. In other words, $\smash{Z^{\lambda, b}_i}$ is the righthand side of the second equality of \eqref{e.boundary problem first mention} with $(\lambda+x,i)$ in place of $(\lambda+x,1)$, i.e.,
\begin{equation}\label{e.Z defn in intro}
Z^{\lambda,b}_i(x) = \sup_{j\geq i}\left\{b^{\lambda}_j + \cP^\lambda_{j\to i}(x)\right\}.
\end{equation}
With $Z^{\lambda,b}_i$ defined, it should be plausible that we can write $\S(y_b,\lambda+x)$ as an LPP problem with boundary data driven by the functions $\smash{\{Z^{\lambda,b}_i\}_{i\in\N}}$ analogous to \eqref{e.boundary problem first mention}. That is, we have that for any $0\leq y\leq x,$
\begin{equation}\label{e.S with Z boundary data intro}
\S(y_b,\lambda+x) = \sup_{i\in\N} \left\{Z^{\lambda,b}_i(y) + \cP[(y,i) \to (x,1)]\right\}.
\end{equation}
The maximizing index $i$ (which is a.s. finite) in the supremum can be thought of as the index of the horizontal line which the infinite geodesic (corresponding to the starting point $y_b$ in $\S$) with endpoint $(\lambda+x,1)$  is present on at the vertical line $\lambda+y$.

Now while the above equation has boundary data on the left, i.e., on the vertical line $\lambda+y$, one can also write down a representation where we make use of $Z^{\lambda, b}_i(y)$ as a process in $y$ so that we have boundary data on the bottom instead of to the left. Indeed, since $\S$ can be heuristically thought of as the LPP weight to the top line in the environment given by $\cP$, while $Z^{\lambda,b}_2=Z^{\lambda,b}_2(\cdot)$ is the LPP weight to the second line, this representation would give $\S(y_b,x)$ as an LPP problem across the two-line environment  $(\smash{\cP^\lambda_1, Z^{\lambda,b}_2})$ on $[0,\infty)$. 

A moment's thought, however, reveals a slight complication: the infinite geodesic could be below line 1 or \emph{at} line 1 at the vertical line $\lambda$. If it is below $1$, then it will enter the second line at a coordinate in $[\lambda,\infty)$, and its contribution to $\S(y_b, \lambda+x)$ will be captured by $Z^{\lambda,b}_2$. But if it is at line $1$ at $\lambda$, then it has entered and exited line $2$ before $\lambda$, and so $\smash{Z^{\lambda,b}_2}$ on $[0, \infty)$ will not be able to capture its contribution. In the latter case we will need an extra term involving the boundary data at $(\lambda,1)$, namely $b^\lambda_1$. The above can now be encoded in the following, somewhat technical looking, formula
$$\S(y_b, \lambda+x) = \sup_{0\leq s\leq x}\left(Z^{\lambda, b}_2(s) + \cP_1^\lambda(x) - \cP_1^\lambda(s)\right)\vee \left(b^\lambda_1 + \cP_1^\lambda(x)\right)$$
The second term is for the case of entry of infinite geodesic at line 1 at vertical line $\lambda$, and the first term for entry of the infinite geodesic into line 2 at or after  $\lambda$. Note that, at $x=0$, the correct expression is $\S(y_b, \lambda) = b^\lambda_1$, i.e., the second term in the previous display; which is not surprising since the infinite geodesic ending at $(\lambda,1)$ is at the top line at $\lambda$ by definition.

With this representation we can obtain a formula for $\wdf(\lambda+x)$ which is similar to \eqref{e.easy two case} obtained in the special case (though note that this formula is for $\wdf$ itself and not an increment of it): assuming for simplicity that the infinite geodesic to $(\lambda+x,1)$ enters the second line on or after $\lambda$,
%
\begin{equation}\label{e.wdf rep in intro}
\wdf(\lambda+x) = \max_{0\leq s\leq x}\left(Z^{\lambda,b}_2(s) - \cP_1^\lambda(s)\right) - \cP_1(\lambda).
\end{equation}
Of course, our statement is for an increment of $\wdf$, as, for example, $\wdf(\lambda+x)$ may be negative and so not comparable to Brownian local time. However, let us first proceed assuming \eqref{e.wdf rep in intro} is the equality we work with, in order to see some of the proof ingredients. Minor annoyances like the dangling term of $\cP_1(\lambda)$ will also disappear when the increment is taken.

Now, if we knew that $(\cP^\lambda_1, Z^{\lambda,b}_2)$ were absolutely continuous to independent rate two Brownian motions on an interval $[0,T]$ for some $T > 0$, then the difference inside the maximum on the righthand side of \eqref{e.wdf rep in intro}, as a process on $[0,T]$, would be comparable to a rate four Brownian motion (except for a minor issue that the difference does not equal zero at $s=0$; this issue will be handled when we consider the increment) and we would be done by the same reasoning as in the special case above.
However, in this case, we cannot obtain absolute continuity to independent Brownian motions on the entire interval $[0,T]$  for $\smash{(\cP^\lambda_1, Z^{\lambda,b}_2)}$, unlike $\smash{(\cP^\lambda_1, \cP^\lambda_{2})}$. Instead, we will need to exclude a neighbourhood of zero, the reason for which will be indicated at the end of this subsection.

To see why absolute continuity of Brownian motion is expected, recall that we already know from the Brownian Gibbs property that, for any $k\in\N$, the top $k$ lines of $\cP$ are jointly absolutely continuous to $k$ independent Brownian motions on any compact interval. Given this and the definition~\eqref{e.Z defn in intro} of $\smash{Z^{\lambda,b}_i}$, a similar comparison between $\smash{Z^{\lambda,b}_i}$ and Brownian motion would be accomplished by knowing two things: first, that there almost surely exists a finite maximizing index for the RHS of \eqref{e.Z defn in intro}, and second, that LPP problems with arbitrary boundary data in environments of a finite number of independent Brownian motions enjoy an absolute continuity comparison to Brownian motion. 

Versions of these facts were proven in \cite{sarkar2020brownian} and form a starting point for our arguments. We mention an important detail of the second fact concerning absolute continuity of general boundary condition Brownian LPP problems to Brownian motion, namely that the comparison does \emph{not} hold on $[0,T]$, but only on $[\varepsilon, T]$ for any $\varepsilon>0$. As we mentioned, we will explain the necessity of this point shortly.



With these two facts, we know that $(\cP^\lambda_1, Z^{\lambda,b}_2)$ is absolutely continuous to two independent rate two Brownian motions on $[\varepsilon, T]$ for any $0<\varepsilon<T$, and so $\smash{Z^{\lambda,b}_2-\cP^\lambda_1}$ is absolutely continuous to rate four Brownian motion on $[\varepsilon,T]$. Since we may select $\lambda$ and $T$, we can ensure that $\lambda+\varepsilon< c < d < \lambda+T$, where $[c,d]$ is the interval of interest from Theorem~\ref{t.patchwork}, and so the Brownian comparison of $\smash{Z^{\lambda,b}_2-\cP^\lambda_1}$ from only $\varepsilon$ onwards is sufficient for our application.

In fact, we will take $\lambda$ to be sufficiently negative so that it is much smaller than $c$. This will be useful because, in that case, it is very likely that the infinite geodesic to $(x,1)$ for any $x\in[c,d]$ will be at line two or lower at $\lambda$. This is simply based on the idea that being forced to remain on the top line on the entire interval $[\lambda, c]$ will prevent the geodesic from following the parabolically curved shape we expect (see the discussion before Section~\ref{s.intro.boundary data}). The precise consequence that we require, that \eqref{e.wdf rep in intro} holds,  will be a part of the proof.

Let $X$ be a process on $[0,\infty)$ such that $X-X(0)$ is absolutely continuous to Brownian motion on $[0,T]$ for all $T>0$, and let $M^X$ be running maximum process of $X-X(0)$, and $M^B$ be the analogue for Brownian motion. 
The proof of Theorem~\ref{t.patchwork} is complete if we can say that $M^X(\cdot)-M^X(c)$ on $[c,d]$ is absolutely continuous to $M^B(\cdot) - M^B(c)$ on $[c,d]$, since the latter has the same distribution as the increment of Brownian local time by L\'evy's identity; we will simply take $X = \smash{Z^{\lambda,b}_2-\cP^\lambda_1}$ on the high probability event that \eqref{e.G' defn} holds. The claim is essentially an immediate consequence of the absolute continuity hypothesis on $X$, and completes the overview of Theorem~\ref{t.patchwork}. 

Next we say a few words on why LPP values with boundary data in Brownian environments are absolutely continuous to Brownian motion only away from zero, the source of the same statement for $Z^{\lambda,b}_2$. Consider the case where the boundary data is identically zero, and so we are considering a comparison of $x\mapsto B[(0,j)\to(x,i)]$ to Brownian motion, where, analogous to the notation for LPP values in $\cP$, $B[(0,j)\to(x,i)]$ represents the LPP value from $(0,j)$ to $(x,i)$ in an environment of independent Brownian motions. For simplicity, we may take $j=2$ and $i=1$. Then we see that
$$B[(0,2) \to (x,1)] = B_1(x) + \max_{0\leq s\leq x} \left(B_2(s) - B_1(s)\right),$$
where $B_1$ and $B_2$ are independent Brownian motions. The drift factor $\max_{0\leq s\leq x} \left(B_2(s) - B_1(s)\right)$ is the basic source of the singularity ``at zero'' causing the mean of the displayed process to be strictly positive, unlike Brownian motion. Indeed, consider an empirical average of scaled increments of the form $\sum_{i=1}^n (X(2^{-i}) - X(2^{-i+1}))2^{(i-1)/2} n^{-1}$. When $X$ is Brownian motion, this will almost surely converge to zero (by the strong law of large numbers), but when $X = B[(0,2)\to(\cdot,1)]$, the same statistic's $\limsup$ will be strictly positive almost surely (it is a constant by applying Blumenthal's zero-one law to the underlying Brownian motions).

\subsubsection{Understanding geodesic geometry at a random time}

The proof of Theorem~\ref{t.local limit} requires a detailed understanding of the geodesic geometry. A particularly elegant aspect of our proof is that it ultimately reduces to the well-known fact that two-dimensional Brownian motion avoids any given point in the plane forever with probability one.

{While the theorem is stated for four different choices of $\tau,$ in this proof overview we will mainly discuss the case where $\tau=\tau_\lambda$, which captures the main conceptual points common to all the cases. We will say a few words at the end about the other choices of $\tau$.}

First, observe that Theorem~\ref{t.local limit} concerns behaviour of $\wdf$ at a \emph{random} time $\tau_\lambda$, which is almost surely greater than $\lambda$. So a first approach to proving Theorem~\ref{t.local limit} might proceed by arguing as above to conclude from \eqref{e.wdf rep in intro} that $\wdf(\tau_\lambda)$ equals
 $\smash{Z^{\lambda,b}_{2}(\tau_\lambda) - \cP^{\lambda}_{1}(\tau_\lambda)} - \cP_1(\lambda)$. 
While it is indeed possible that one could obtain Theorem~\ref{t.local limit} by an analysis relying directly on the absolute continuity of $\smash{Z^{\lambda,b}_{2} - \cP^{\lambda}_{1}}$ to Brownian motion, the approach chosen in the paper yields additional information on the geometry of geodesics at random times and on the joint distribution of $\smash{Z^{\lambda,b}_i}$ for different $i$, which is of independent interest.

Consider the infinite geodesic in $\cP$ as its endpoint moves from $\tau_\lambda$ slightly forward. More precisely, we look at the location (i.e., line index) of the geodesic at $\tau_\lambda$ as the endpoint moves further. In other words, we need to use the formula \eqref{e.S with Z boundary data intro} with $y=\tau_\lambda-\lambda$. But now the issue is that behaviour at $\tau_\lambda$ of $\wdf$ corresponds to behaviour of an expression of the form \eqref{e.wdf rep in intro} (again with $\tau_\lambda$ in place of $\lambda$) \emph{at the origin}. But as we saw expressions like $\smash{Z^{\lambda,b}_2-\cP^\lambda_1}$ are not absolutely continuous to Brownian motion at the origin!

If we knew that we could replace $\smash{Z^{\lambda,b}_2}$ by $\smash{\cP^\lambda_2}$, it is an easy argument, using the scale invariance and independence across scales of Brownian motion, that the local limit at the origin of an expression like $\max_{0\leq s\leq x}\left(\cP^{\lambda}_{2}(s) - \cP^{\lambda}_{1}(s)\right)$ is Brownian local time; see Lemma~\ref{l.local limit BM} and Corollary~\ref{c.local limit local time}. We can handle the fact that we are looking at $\smash{\cP_2^\lambda}$ and $\smash{\cP_1^\lambda}$ at a stopping time of $\cP$ by making use of a stronger version of the Brownian Gibbs property which works for spatial analogues of stopping times for Markov processes; $\tau_\lambda$ is such a random quantity. So $\cP_2^\lambda$ and $\cP_1^\lambda$ look Brownian even at $\tau_\lambda$.

So our proof consists of showing that in the immediate right-neighbourhood of $\tau_\lambda$, we can in fact replace $Z^{\lambda,b}_2$ by $\cP^\lambda_2$; in terms of the infinite geodesic, it is at line 2 (and not lower) at $\tau_\lambda$. 

Ignoring certain technical details, such as setting up the appropriate right-continuous filtration, we review the the basic argument here.  Note that the infinite geodesic (corresponding to the starting point $y_b$) ending at $(\tau_\lambda, 1)$ will certainly be at the top line at $\tau_\lambda$ (note that this does not preclude the geodesic having a vertical stretch of jumps from $(\tau_\lambda,k)$ to $(\tau_\lambda,1)$ for some $k>1$), and that we may think of the geodesic corresponding to $y_a$ as staying on the top line at all times (the geodesic interpretation of the convenient fact that $\S(0,x) = \cP_1(x)$). Now, because $\tau_\lambda$ is a point of increase of $\wdf$, the $y_b$-geodesic must switch lines immediately after $\tau_\lambda$; if not, both geodesics use the top line in that neighbourhood and $\wdf$ would be a constant.

To say that we can replace $Z^{\lambda,b}_2$ by $\cP^\lambda_2$ in the right-neighbourhood of $\tau_\lambda$, we need to say that that the $y_b$-geodesic jumps to line 2 and not lower.

The above argument applies not only to $\tau_\lambda$, but to any stopping time which is almost surely a point of increase of $\wdf$ and is greater than $\lambda$. Dropping the condition of being greater than $\lambda$ is a simple technical task, which then takes care of the case where $\tau = \rho^h$ (hitting time of $h\in\R$) in Theorem~\ref{t.local limit}. The final case of $\xi_{[c,d]}$ (an independent uniform sample from $\wdf$ on $[c,d]$) is slightly more complicated as it is not a stopping time. For this, we rely on a representation of it as the hitting time $\smash{\rho_c^U}$, i.e., the hitting time of $\wdf(c) + U$, with $U$ a uniform random variable on $[0,\wdf(d)-\wdf(c)]$, which can then be decomposed into a mixture of stopping times $\rho_c^h$.

Now we turn to explaining why the $y_b$-geodesic must jump to line 2 and not lower.

\subsubsection{{Boundary data non-intersection properties}} To ensure this we analyze the boundary data processes $\smash{\{Z^{\lambda,b}_j\}_{j\in\N}}$ at the random location $\tau_\lambda$. Observe first from the definition~\eqref{e.Z defn in intro} that $Z^{\lambda,b}_1(\tau_\lambda - \lambda)\ge Z^{\lambda,b}_2(\tau_\lambda-\lambda)\ge\ldots $. If we knew the strict inequality $Z^{\lambda,b}_2(\tau_\lambda-\lambda) > Z^{\lambda,b}_3(\tau_\lambda-\lambda)$, the continuity of $\cP$ will imply from \eqref{e.S with Z boundary data intro} that 
$$Z^{\lambda,b}_2(\tau_\lambda-\lambda)+\cP[(\tau_\lambda,2)\to(\tau_\lambda+\varepsilon x,1)] > Z^{\lambda,b}_3(\tau_\lambda-\lambda) + \cP[(\tau_\lambda,3)\to(\tau_\lambda+\varepsilon x,1)]$$
for all small $\epsilon$ and $x$ in a compact set, and so the $y_b$-geodesic must be at line two for all such small $\epsilon$ and bounded $x$.

We also know that $Z^{\lambda,b}_1(\tau_\lambda-\lambda)= Z^{\lambda,b}_2(\tau_\lambda-\lambda)$, for otherwise the $y_b$-geodesic could not jump to line 2 at $\tau_\lambda$. So our proof has essentially reduced to showing that $\smash{Z^{\lambda,b}_1(\tau_\lambda-\lambda) > Z^{\lambda,b}_3(\tau_\lambda-\lambda)}$ almost surely. So it suffices to show that, almost surely for all $x\geq 0$, $\smash{Z^{\lambda, b}_1(x) > Z^{\lambda, b}_3(x)}$.

%

We saw earlier that $\smash{Z^{\lambda, b}_3}$ is absolutely continuous to Brownian motion away from the origin. We can now simply consider the three line ensemble given by  $\smash{(\cP_1^\lambda,\cP_2^\lambda, \smash{Z^{\lambda, b}_3})}$. As before, since $\smash{Z^{\lambda, b}_3}$ is heuristically an LPP problem to the third line, we can write $\smash{Z^{\lambda, b}_1}$ as an LPP problem in the mentioned three line ensemble. 
Working with the expressions for LPP values in \emph{two}-line environments, for any $i\in\N$, $\smash{Z^{\lambda, b}_i}$ can essentially be expressed as $\cP_i^\lambda(\cdot)$ reflected off of $\smash{Z^{\lambda, b}_{i+1}}$, in the sense of Skorohod reflection (this relation is also known as a Pitman transform, for example in \cite{sarkar2020brownian}); see \cite{warren2007dyson} for work within KPZ on Brownian motions reflecting of Dyson's Brownian motion.

With this, and the absolute continuity away from zero with respect to independent Brownian motions, the event that $Z^{\lambda, 1}(x) = Z^{\lambda, 3}(x)$ for some $x$ can be related to the event that two dimensional Brownian motion hits the origin at some strictly positive time. But of course, it is a classical fact that this event has probability zero, which, by absolute continuity, implies what we need.

These arguments are developed in Section~\ref{s.local limit proof}. 

Next, to contrast with our methods we include a brief discussion of the approach in \cite{basu2019fractal}. 
\subsection{A brief comparison with \cite{basu2019fractal}}\label{s.comparison to bgh}
The paper \cite{basu2019fractal} analyzed Brownian LPP and relied on probability estimates for rare geometric events involving geodesics. Essentially, the points of $\NC(\wdf)$ correspond exactly to the endpoints $x$ such that the geodesics from $(y_a, 0)$ and $(y_b,0)$ to $(x,1)$ are disjoint except for their shared endpoint; see Figure~\ref{f.coalescence}. (Technically, and as mentioned earlier, only the containment of $\NC(\wdf)$ in the latter set was proved in \cite{basu2019fractal}, which sufficed for their purposes, while the equality of the sets was proved later in \cite{bates2019hausdorff}.)

\begin{figure}[h]
\centering
\includegraphics[scale=0.9]{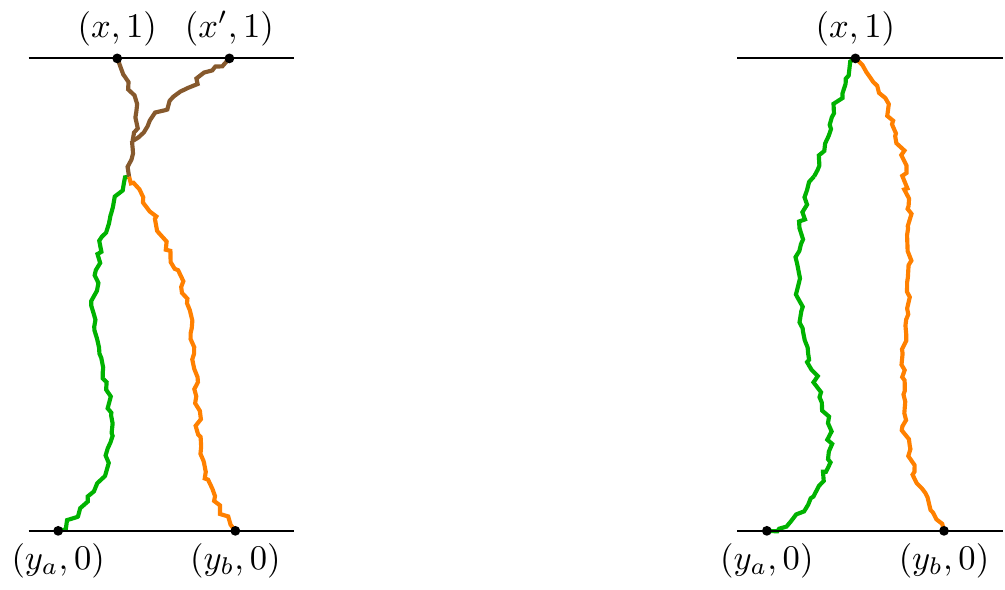}
\caption{A depiction of the two possible behaviours of the geodesics with separate fixed starting points $(y_a,0)$ and $(y_b,0)$ and common ending point $(x,1)$. Left: The two geodesics (green and orange) coalesce, with the common portion shown in brown. When $x$ is varied locally (eg. to $x'$), the brown portion will be modified, but the green and orange portions will remain fixed; since $\wdf(z)$ for $z$ close to $x$ is the difference of the orange and green weights, $\wdf$ is locally constant at $x$. Right: The contrasting situation when the two geodesics remain disjoint till the last instant. Here moving $x$ locally modifies both geodesics in an unshared way, so $\wdf$ does not remain constant.}
\label{f.coalescence}
\end{figure}

Thus, the basic observation of \cite{basu2019fractal} leads to 
considering the probability of having disjoint geodesics, in the prelimiting model of Brownian LPP; such estimates were available from earlier work of Hammond \cite{brownianLPPtransversal}. This gives an upper bound on the Hausdorff dimension of $\NC(\wdf)$. 
The lower bound was obtained by a different argument relying on a quantified form of Brownianity (H\"older regularity) of the weight profile from any given source. 

In short, in contrast to our methods, \cite{basu2019fractal} relies heavily on an analysis of geodesic geometry in Brownian LPP to study $\NC(\wdf)$ directly. Unfortunately, such an approach does not provide any insight about the finer structure of $\wdf$ that we are interested in.

\subsection*{Acknowledgments}
The authors thank Manjunath Krishnapur for originally posing the question this paper addresses, and Alan Hammond for encouraging them to consider the question of the local limit at a typical point picked according to $\wdf$. SG thanks Riddhipratim Basu for useful discussions at the early stages of this project. He is partially supported by NSF grant DMS-1855688, NSF CAREER Award DMS-1945172 and a Sloan Research Fellowship. MH acknowledges the support of NSF grant DMS-1855688.

\section{Parabolic Airy line ensemble, LPP, and  geometry of geodesics}
\label{s.geometry}
This section develops formally the central objects that will be in play throughout the paper. We then derive some useful geometric lemmas. 
We start by introducing the precise definition of the parabolic Airy sheet $\S$. Because this is defined in terms of a last passage problem in a random environment defined by the parabolic Airy line ensemble $\cP$, we first define the former concept and the latter object in Section~\ref{s.lpp and airy}.

\begin{figure}[h]
\centering
\includegraphics[scale=0.8]{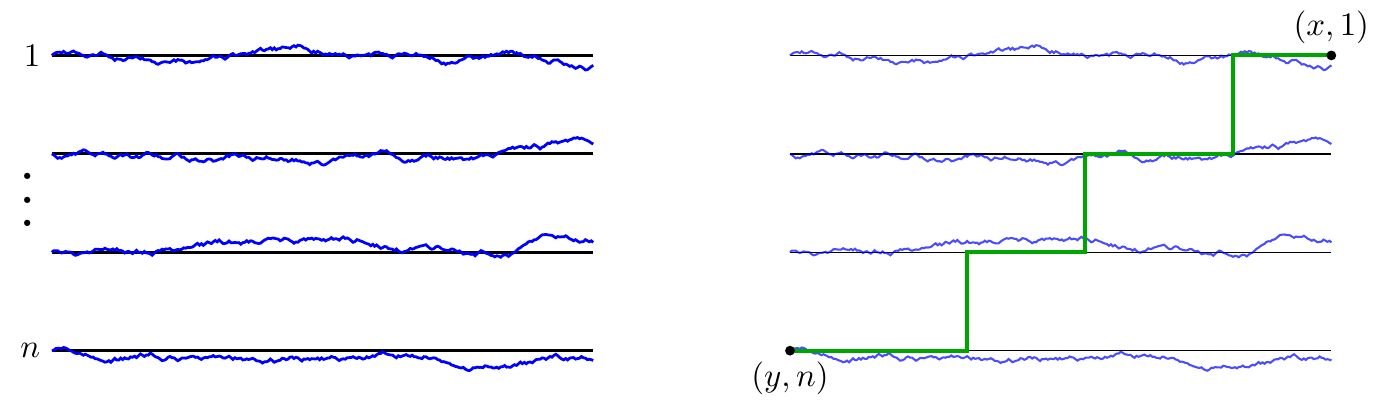}
\caption{Left: An illustration of (a subset of) the environment defined by $f$. The functions $f_i$ corresponding to each line are shown in blue on the corresponding black line for visual clarity; the function values themselves are not necessarily ordered. Right: An up-right path $\gamma$ from $(y,n)$ to $(x,1)$ is shown in green; note that in the formal definition the depicted vertical portions are not part of $\gamma$. The path's weight is the sum of the increments of $f_i$ along the portion of the $i$\textsuperscript{th} line $\gamma$ spends on it.}
\label{f.lpp}
\end{figure}

\begin{notation}
We denote the integer interval $\{i, \ldots, j\}$ by $\intint{i,j}$, and the set $\{1,2, \ldots, \}$ by $\N$. For an interval $I\subseteq \R$, $\mc C(I)$ will be the set of continuous functions $f:I\to\R$. For a finite set $A$, $\#A$ will denote its cardinality. Finally, given a function $f$ defined on a set $A$, we will denote its restriction to a subset $B\subseteq A$ by $f|_B$.
\end{notation}

\goodbreak

\subsection{Last passage percolation and the parabolic Airy line ensemble}\label{s.lpp and airy}

\begin{definition}[Last passage percolation]\label{d.lpp}
Let $f=(f_1,f_2, \ldots) : \N\times\R\to\R$ be a given sequence of continuous functions. We will think of the curves $f_1, f_2, \ldots, $ as lying on horizontal lines which are ordered vertically, indexed such that $f_1$ is the top curve, $f_2$ is below it, and so on; though the function values themselves are not assumed to be ordered. See Figure~\ref{f.lpp}. We will refer to these curves as the \emph{environment} given by $f$, or defined by $f$.

Let $y\leq x$ and $n\geq m$. An \emph{up-right} path from $(y,n)$ to $(x,m)$ is a path which begins at $(y,n)$  and moves rightward, jumping from one line to the next at various times, till it reaches $(x,m)$; see Figure~\ref{f.lpp}. Up-right paths are parametrized by their \emph{jump times} $\{t_i\}_{i=m+1}^{n}$ at which they jump from the $i$\textsuperscript{th} line to the $(i-1)$\textsuperscript{th} line. The \emph{weight} of an up-right path $\gamma$ from $(y,n)$ to $(x,m)$ through $f$ is given by
$$f[\gamma] = \sum_{i=m}^{n} \bigl(f_i(t_i) - f_i(t_{i+1})\bigr),$$
where $t_{n+1} = y$ and $t_m = x$. In other words, $f[\gamma]$ is the sum over $i$ of the increments of $f_i$ over the interval that $\gamma$ spends on the $i$\textsuperscript{th} line. The \emph{last passage value} from $(y,n)$ to $(x,m)$ through $f$ is given by
\begin{equation}\label{e.lpp defn}
f[(y,n)\to(x,m)] = \sup_{\gamma:(y,n)\to(x,m)} f[\gamma],
\end{equation}
where the supremum is over all up-right paths from $(y,n)$ to $(x,m)$. A path which achieves the supremum is called a \emph{geodesic} from $(y,n)$ to $(x,m)$.

Note that the definition \eqref{e.lpp defn} immediately adapts to the case that $f$ has domain $I\times\R$ or $I\times [0,\infty)$ for some finite integer interval $I\subseteq \N$, so long as $n,m\in I$, and, in the latter case, $x,y\geq 0$.
\end{definition}

While there are broad similarities, the heuristic LPP model described in Section~\ref{s.intro} should not be confused for the precise LPP model defined above. 
We emphasize that the directedness constraint in Definition~\ref{d.lpp} imposes the inequality $y\leq x$ and that the paths are not well-defined functions of their height, i.e., the line index.

Having defined the general LPP model, we now move on to formally introducing another central object in this paper, the parabolic Airy line ensemble.

\begin{definition}[Parabolic Airy line ensemble]\label{d.parabolic Airy line ensemble}
The parabolic Airy line ensemble $\cP: \N\times\R\to \R$ is an $\N$-indexed family of random non-intersecting continuous curves, such that the ensemble $\A$ given by $\A_i(x) = \cP_i(x) +x^2$ has finite dimensional distributions determined as follows: for every $m\in\N$ and real $t_1 < \ldots< t_m$, the point process $\{(\A_i(t_j),t_j) : i\in\N, j\in\intint{1,m}\}$ is determinantal with correlation kernel given by the extended Airy kernel $K_\mathrm{Ai}^{\mathrm{ext}}$, where
$$K_\mathrm{Ai}^{\mathrm{ext}}\bigl((x,t); (y,s)\bigr) = \begin{cases}
\int_0^\infty e^{-\lambda(t-s)}\mathrm{Ai}(x+\lambda)\mathrm{Ai}(y+\lambda)\, \mathrm d\lambda & t\geq s\\
-\int^0_{-\infty} e^{-\lambda(t-s)}\mathrm{Ai}(x+\lambda)\mathrm{Ai}(y+\lambda)\, \mathrm d\lambda & t< s;
\end{cases}$$
\end{definition}
here $\mathrm{Ai}$ is the classical Airy function. The reader is referred to \cite{manjunath} for background on determinantal point processes.

The parabolic Airy line ensemble was first constructed in \cite{corwin2014brownian}. That paper also proved $\cP$ enjoys the \emph{Brownian Gibbs property}, a central invariance property which we define next. 

\begin{definition}[Brownian Gibbs property]\label{d.bg}
Let $X:\N\to\R$ be a collection of random continuous non-intersecting curves, and fix $k\in\N$ and $[\ell,r]\subset \R$. Let $\F_{\mathrm{ext}}(k,\ell,r)$ be the $\sigma$-algebra generated by $\{X_i(x) : (i,x)\not\in \intint{1,k}\times [\ell,r]\}$, i.e., everything external to $[\ell,r]$ on the top $k$ curves. $X$ is said to possess the \emph{Brownian Gibbs} property if, conditionally on $\F_{\mathrm{ext}}(k,\ell,r)$, the distribution of $X$ on $\intint{1,  k}\times[\ell,r]$ is that of $k$ independent Brownian bridges $(B_1, \ldots, B_k)$ of rate \emph{two}, with $B_i(x) = X_i(x)$ for $x\in\{\ell,r\}$ and $i\in\intint{1, k}$, conditioned on not intersecting each other or $X_{k+1}(\cdot)$ on $[\ell,r]$.
\end{definition}

\begin{theorem}[Theorem~3.1 of \cite{corwin2014brownian}]
The parabolic Airy line ensemble $\cP$ has the Brownian Gibbs property.
\end{theorem}

We will be making use of the Brownian Gibbs property of $\cP$ quite often; indeed, it is the ultimate source of all the statements we make regarding absolute continuity to Brownian objects, the first of which is the following.

\begin{corollary}\label{c.abs cont}
Let $[\ell, r]\subset \R$ and $k\in \N$. Then $\{\cP_i(\cdot) - \cP_i(\ell): i\in\intint{1,k}\}$, as a process on $[\ell, r]$, is absolutely continuous to the law of $k$ independent Brownian motions of rate two, all started at $(\ell,0)$, on $[\ell, r]$.
\end{corollary}

This proof is the same as the one given for \cite[Proposition~4.1]{corwin2014brownian}, where the $k=1$ case of Corollary~\ref{c.abs cont} is proved, and we reproduce it here.

\begin{proof}[Proof of Corollary~\ref{c.abs cont}]
Applying the Brownian Gibbs property on $\intint{1,k}\times[\ell, r]$ tells us that, conditionally on $\F_{\mathrm{ext}}(\ell,r,k)$, $\{\cP_i(\cdot) - \cP_i(\ell): i\in\intint{1,k}\}$ is absolutely continuous to the law of $k$ independent Brownian bridges of rate two. To move from here to Brownian motions, we must know that $(\cP_i(r) - \cP_i(\ell))_{i=1}^k$ has law which is absolutely continuous to product Lebesgue measure. This is implied by applying the Brownian Gibbs property on $\intint{1,k}\times[\ell, 2r]$.
\end{proof}

We will often make use of this corollary without explicitly referring to it by name.

\subsection{The parabolic Airy sheet}
The parabolic Airy sheet was proved to exist in \cite{dauvergne2018directed}. Its definition, as given in \cite[Definition~8.1]{dauvergne2018directed}, is in terms of a last passage problem in the parabolic Airy line ensemble. First, for $k\in\N$ and $y>0$, let 
$$(y)_k = \left(-\left(k/2y\right)^{1/2}, k\right).$$

\begin{definition}[Parabolic Airy sheet]\label{d.airy sheet}
The parabolic Airy sheet $\S:\R^2\to\R$ is a continuous process with the following two properties. 

\begin{enumerate}
	\item $\S(\cdot + z, \cdot + z) \stackrel{d}{=} \S(\cdot, \cdot)$ for each $z\in\R$, where the equality is in distribution of processes.

	\item $\S$ can be coupled with the parabolic Airy line ensemble $\cP$ so that $\S(0,\cdot) = \cP_1(\cdot)$ and, almost surely for all $x,y,z\in\Q$ with $y>0$, there exists a random constant $K_{x,y,z}$ such that, for all $k\geq K_{x,y,z}$,
	$$\S(y,z) - \S(y,x) = \cP\bigl[(y)_k \to (z,1)\bigr] - \cP\bigl[(y)_k \to (x,1)\bigr].$$
\end{enumerate}
\end{definition}

The existence of the parabolic Airy sheet was proven in \cite{dauvergne2018directed} via Brownian last passage percolation. To state this, let $B:\N\times\R\to\R$ be a $\N$-indexed family of independent two-sided Brownian motions of rate one.

\begin{theorem}[Theorem~1.3 of \cite{dauvergne2018directed}]\label{t.airy sheet}
The Airy sheet exists and is unique in law. Further,
$$\S(y,x) = \lim_{n\to\infty} n^{-1/3}\left(B[(2yn^{2/3}, n) \to (n+2xn^{2/3}, n)] - 2n - 2(x-y)n^{2/3}\right),$$
where the limit is in distribution, in the topology of uniform convergence on compact sets.
\end{theorem}

Note that the stationarity asserted in (i) of Definition~\ref{d.airy sheet} above allows us to assume without loss of generality that $y_a = 0$ (and so $y_b > 0$) in the definition \eqref{e.W defn} of $\wdf$, and we do so for the rest of the paper. This is useful as it gives us access to the formula in (ii) of Definition~\ref{d.airy sheet}, which requires the initial point $y$ to be positive, and also gives that $\S(y_a, x) = \cP_1(x)$.

\subsection{LPP boundary values in \texorpdfstring{$\cP$}{parabolic Airy line ensemble}}\label{s.geometry.lpp boundary values}
We now set up the framework of LPP models with  \emph{boundary data} which, as indicated in Section \ref
{s.intro.proof strategy}, will play a vital role in many of our arguments. This was introduced in \cite{sarkar2020brownian}. However, before formally defining things, we first discuss the various issues one encounters in making certain definitions and how they are addressed. 

Observe that item (ii) in Definition~\ref{d.airy sheet} is suggestive of the possibility that $\S(y,x)$ itself, for $y>0$, can be written as a limiting LPP value in the parabolic Airy line ensemble, without having to resort to an expression which only involves a difference. This is not directly true, as $\cP[(y)_k\to(x,1)]$ will diverge to $\infty$ as $k\to\infty$. Instead, it is conjectured (see \cite[Conjecture~14.2]{dauvergne2018directed}) that there exists a deterministic function $a: \R^+\times \N\to \R$ such that, for every fixed $y>0$, almost surely
\begin{equation}\label{e.conjectures S formula}
\S(y, 0) = \lim_{k\to\infty} \bigl(\cP[(y)_k\to(0,1)] - a(y,k)\bigr).
\end{equation}
However, something of this sort appears quite difficult to prove. The main utility for us of such an expression is that it would allow the use of reasoning about geodesics in analyzing $x\mapsto \wdf(x) = \S(y_b, x) - \S(y_a,x)$, for which the LPP description of Definition~\ref{d.airy sheet} is not available as it is not a difference of the parabolic Airy sheet at different ending values with common starting point.

Imagine for a moment that we could express $\S(y,x)$ as the weight of an infinite geodesic through the parabolic Airy line ensemble ending at $(x,1)$ with a formula like \eqref{e.conjectures S formula}. Then, by considering the index $i$ of the line this geodesic is at at a location $\lambda \leq x$, we see that it would be possible to write $\S(y,x)$ as
\begin{equation}\label{e.proxy formula}
\S(y,x) = \sup_{i\in\N} \bigl(\cP[y\to (\lambda, i)] + \cP[(\lambda, i)\to (x,1)]\bigr),
\end{equation}
where $\cP[y\to (\lambda, i)]$ represents a renormalized form of the weight of the infinite geodesic ending at $(\lambda, i)$, akin to the righthand side of \eqref{e.conjectures S formula} with $i$ replacing 1. Note that, while the first term in the supremum is currently difficult to define directly, the second is a standard finite LPP problem. The collection $\{\cP[y\to(\lambda, i)] : i\in\N\}$ can be thought of as \emph{boundary data}, with respect to which the finite LPP problem $\cP[(\lambda, i)\to (x,1)]$ is considered.

The crucial observation, however, is that even in the absence of \eqref{e.conjectures S formula}, equation \eqref{e.proxy formula} is also useful for analyzing $\wdf$, for the expression for $\S(y_b,x)$ involves a finite LPP problem in $\cP$, while $\S(y_a,x) = \cP_1(x)$ (since $y_a=0$).

Thus it remains to look for a well-defined quantity that can take the place of $\cP[y\to (\lambda, i)]$ in the above formula and play the role of boundary data. The problem with defining $\cP[y\to (\lambda, i)]$ directly was that $\cP[(y)_k\to(\lambda,i)]$ diverges to infinity and we do not currently have a statement of the form \eqref{e.conjectures S formula} that would have allowed us normalize  the latter by a quantity like $a(y,k)$ to obtain a well defined limit. 

To get around this, with a device similar in spirit to how differences were considered in Definition \ref{d.airy sheet}, \cite{sarkar2020brownian} considered the quantity
\begin{equation}\label{e.proxy boundary}
\cP[(y)_k\to(\lambda,i)] - \cP[(y)_k\to(\lambda,1)].
\end{equation}
Owing to the fact that the starting points are the same,
this quantity's limit can be shown to exist more easily, and it is finite. This is essentially because the finite geodesics corresponding to the two terms share the same path from their starting point to a certain height, i.e., \emph{coalesce}, in a uniform way, only to separate at a unit order distance away from the destination points $(\lambda,1)$ and $(\lambda,i)$. This makes the difference of these geodesics' weights be of unit order.

Thus the difference in \eqref{e.proxy boundary} is essentially a difference of LPP problems from the same random point, which can for example be taken to be the point of separation whose depth, as mentioned, is uniformly bounded in $k$.

{In spite of the issues with defining its weight directly}, the notion of an infinite geodesic through $\cP$ can in fact be made precise, and this is the approach \cite{sarkar2020brownian} adopts. They do this by carefully considering the limit of the finite geodesics mentioned in the previous paragraph. We shall not require that and choose simply to work with finite LPP problems with boundary data, as it avoids some of the technical issues in defining the infinite geodesic. 
However, we do make use of a few results of \cite{sarkar2020brownian} which are proved via arguments involving infinite geodesics in order to ensure the existence of the boundary data and that it satisfies a relation involving $\S$ as in \eqref{e.proxy formula} which will be discussed next.

We now introduce the precise objects. We will sometimes need to consider starting points $y>0$ different from $y_b$, and so use notation that allows for this. We fix $\lambda\in\R$ for the rest of this section and, for $i\in\N$, define
\begin{equation}\label{e.a b definition}
\begin{split}
%
b^{\lambda, y}_{i} &= \lim_{k\to\infty}\left(\cP\bigl[(y)_k \to (\lambda,i)\bigr] - \cP\bigl[(y)_k \to (\lambda,1)\bigr] + \S(y, \lambda)\right).
\end{split}
\end{equation}
(The term $\S(y,\lambda)$ allows a formula like \eqref{e.proxy formula}, see Lemma \ref{l.Airy sheet to boundary LPP problem} below.) It is proved in \cite[Theorem 3.7]{sarkar2020brownian} that, for $\lambda=0$ and every $i\in\N$, these limits exist and are finite almost surely; in fact, there exists a random integer $K_{i}$ (which corresponds to the depth at which the geodesics from $(y)_k$ to $(\lambda, i)$ and $(\lambda,1)$ coalesce) such that the limiting values in \eqref{e.a b definition} are achieved for all $k\geq K_{i}$. 

While \cite{sarkar2020brownian} states the existence of the limits \eqref{e.a b definition} only for $\lambda= 0$, their argument applies verbatim for any fixed $\lambda\in\R$. However, this does not imply the existence of a boundary data \emph{process} in $\lambda$ which will be needed for many of our arguments, and we will introduce it shortly.

Next, we define some notation for the finite LPP value term in \eqref{e.proxy formula}. Recall from \eqref{comnot} the functions $\cP^{\lambda}_{i\srt 1}:[0,\infty)\to \R$ given by
\begin{equation}\label{e.R definition}
\cP^{\lambda}_{i\srt 1}(x) = \cP[(\lambda,i)\to(\lambda+x,1)].
\end{equation}
The following lemma is the rigorous analogue of \eqref{e.proxy formula} with the boundary data $\{b^{\lambda, y}_i\}_{i\in\N}$. It was proved in \cite{sarkar2020brownian}, again in the $\lambda = 0$ case, but the argument applies to $\lambda\in\R$.

\begin{lemma}[Lemma 3.10 of \cite{sarkar2020brownian}]\label{l.Airy sheet to boundary LPP problem}
Fix $\lambda\in\R$. For all $x\geq 0$,
\begin{equation}\label{e.airy sheet boundary LPP problem}
\begin{split}
%
\S(y,\lambda+x) &= \sup_{i\in\N}\left\{b^{\lambda, y}_i + \cP^{\lambda}_{i\srt 1}(x)\right\}.
\end{split}
\end{equation}
Further, the supremums are almost surely achieved at a finite value of $i$.
\end{lemma}

That the supremums in \eqref{e.airy sheet boundary LPP problem} are attained will be important for us, and we will need to understand properties of the supremum-achieving indices. This will require arguments involving the geometry of the geodesics for the finite LPP problems encoded by $\cP^\lambda_{i\srt 1}$, and the next section is devoted to developing these arguments.

\subsection{Geometry of geodesics}\label{s.geometry of geodesics subsection} 

For $\lambda\in\R$ and $y>0$ fixed, and $x\geq 0$, let $i_y^\lambda(x)$ be the minimum $i$ which achieves the supremum in the first equation of \eqref{e.airy sheet boundary LPP problem}. We will need the following properties.

\begin{lemma}\label{l.i j properties}
For each $0<y<y'$, the following hold almost surely: (i) $i^\lambda_y(x)\leq i^\lambda_{y'}(x)$ for all $x\geq 0$ and (ii) $i^\lambda_y$ is left-continuous and non-decreasing.
\end{lemma}

The intuition behind Lemma~\ref{l.i j properties} is that $\smash{i^\lambda_y(x)}$ and $\smash{i^\lambda_{y'}(x)}$ should be thought of as the line index at $\lambda$ of the infinite geodesic from $y$ to $(\lambda+x,1)$ and $y'$ to $(\lambda+x,1)$ respectively. Since $y < y'$, by planarity, the first infinite geodesic should be to the left of the second, which is encoded by (i) above. Similarly, as $x$ increases, the geodesic corresponding to it should move to the right, which is encoded by (ii). 

On a technical level, Lemma~\ref{l.i j properties} is a consequence of certain monotonicity properties of $\{b^{\lambda,y}_i\}_{i\in\N}$ which encode the planarity relations. One easy to see relation which will have significance later is that, for fixed $y>0$, $b^{\lambda, y}_{i+1} \leq b^{\lambda, y}_i$. This is in fact immediate from  the inequality $\cP[(y_a)_k\to(\lambda,i+1)] \leq \cP[(y_a)_k\to(\lambda, i)]$ for all $k$. The latter follows by considering the geodesic corresponding to the lefthand side and extending it to jump to line $i$ at the last instant. The extended path is in the set of paths that $\cP[(y_a)_k\to(\lambda,i)]$ is maximizing over, which gives the inequality.

However, we will not need the monotonicity of $\{b^{\lambda, y}_i\}_{i\in\N}$ for fixed $y$ to prove parts (i) and (ii) of Lemma~\ref{l.i j properties}. Instead, we will require the monotonicity of the \emph{difference} at $y$ and $y'$, i.e., that we have $b^{\lambda, y}_{i+1}-b^{\lambda, y'}_{i+1} \leq b^{\lambda, y}_{i}-b^{\lambda, y'}_{i}$, which is a consequence of a simple planarity argument that we record next.

\begin{lemma}[Crossing lemma]\label{l.crossing}
Let $f:\N\times\R\to\R$ be any sequence of functions. Let $n_1\geq m_1$, $n_2\geq m_2$, $n_1\leq n_2$, and $m_1\leq m_2$. Also let $y_1\leq x_1$, $y_2\leq x_2$, $x_1\leq x_2$, and $y_1\leq y_2$. Then we have
$$f[(y_1,n_1)\to(x_1,m_1)] + f[(y_2,n_2)\to(x_2,m_2)] \geq f[(y_1,n_1)\to(x_2,m_2)] + f[(y_2,n_2)\to(x_1,m_1)].$$
\end{lemma}

This lemma has appeared many times in the LPP literature, see for instance \cite{basu2019fractal,dauvergne2018directed,balazs2020non,cator2019attractiveness}. See also Figure~\ref{f.crossing} for a visual aid for the proof.

\begin{figure}[h]
\centering
\includegraphics[scale=1]{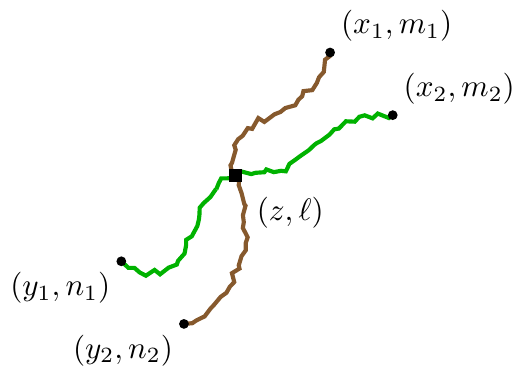}
\caption{An illustration of a choice of starting and ending points which ensures, by planarity, that the geodesics between those points must intersect. The intersection, depicted by a square, is at $(z,\ell)$. By following the brown geodesic from $(y_2, n_2)$ to $(z,\ell)$, and then following the green geodesic from $(z,\ell)$ to $(x_2, m_2)$, we obtain a path from $(y_2,n_2)$ to $(x_2,m_2)$ whose weight can be at most $f[(y_2,n_2) \to (x_2,m_2)]$, and similarly for from $(y_1,n_1)$ to $(x_1,m_1)$.}
\label{f.crossing}
\end{figure}

\begin{proof}[Proof of Lemma~\ref{l.crossing}]
By the positioning of the points assumed and planarity, we have that any geodesic from $(y_1,n_1)$ to $(x_2,m_2)$ must cross any geodesic from $(y_2,n_2)$ to $(x_1,m_1)$ at some point $(z,\ell)$, with $z\in[y_2,x_1]$ and $m_2\leq \ell\leq n_1$. This implies
\begin{align*}
f[(y_1,n_1)\to(x_2,m_2)] + f[(y_2,n_2)\to(x_1,m_1)] &= f[(y_1,n_1)\to(z,\ell)] + f[(z,\ell) \to (x_2,m_2)]\\
&\quad + f[(y_2,n_2)\to(z, \ell)] + f[(z, \ell)\to (x_1,m_1)]. 
\end{align*}
But we have that $f[(y_1,n_1)\to(z, \ell)] +f[(z,\ell)\to (x_1,m_1)] \leq f[(y_1,n_1)\to (x_1,m_1)]$ as the lefthand side is the weight of a particular up-right path from $(y_1,n_1)$ to $(x_1,m_1)$. We have a similar inequality with $y_2, n_2, x_2, m_2$. Applying these inequalities to the last display completes the proof of Lemma~\ref{l.crossing}.
\end{proof}

\begin{corollary}\label{c.a-b monotonicity}
Suppose $j\geq i$ and $y'> y > 0$. Then we have $b^{\lambda, y'}_j - b^{\lambda, y}_j \geq b^{\lambda, y'}_i-b^{\lambda, y}_i$.
\end{corollary}

\begin{proof}
Recalling the definition \eqref{e.a b definition} of $b^{\lambda, y}_j$, it is enough to show that, for all $k\in\N$,
$$\cP[(y')_k\to (\lambda,j)] - \cP[(y)_k\to (\lambda,j)] \geq \cP[(y')_k\to (\lambda,i)] - \cP[(y)_k\to (\lambda,i)]$$
and then take the limit as $k\to\infty$. This last inequality is an immediate consequence of Lemma~\ref{l.crossing} since $y \leq y'$ implies $(y)_k\leq (y')_k$.
\end{proof}

\begin{remark}
It is not the case that $b^{\lambda, y'}_j - b^{\lambda, y}_j > b^{\lambda,y'}_i - b^{\lambda,y}_i$ almost surely, i.e., strict inequality does not always hold, though it may hold with positive probability. Indeed, observe from the proof of Lemma~\ref{l.crossing} that equality holds in that lemma when the geodesic from $(y_1,n_1)$ to $(x_1, m_1)$ coalesces with the geodesic from $(y_2,n_2)$ to $(x_2,m_2)$, i.e, when those two geodesics are not disjoint. This does occur for finite geodesics from $(y)_k$ and $(y')_k$ in $\cP$ with positive probability, and the coalescence asserted via the existence of $K_i$ in the definition \eqref{e.a b definition} of $b^{\lambda,y}_i$ implies that equality can hold in the limit as $k\to\infty$ as well with positive probability.
\end{remark}

The outstanding proof of Lemma~\ref{l.W increasing} is also an immediate consequence of Lemma~\ref{l.crossing}:

\begin{proof}[Proof of Lemma~\ref{l.W increasing}]
We have to show that $\S(y_b, x_1) - \S(y_a, x_1) \leq \S(y_b, x_2) - \S(y_a, x_2)$ whenever $x_1 \leq x_2$, i.e.,
$$\S(y_b, x_2) + \S(y_a, x_1) \geq \S(y_b, x_1) + \S(y_a, x_2).$$
The analogous inequality for Brownian LPP follows from Lemma~\ref{l.crossing}, and the convergence of Brownian LPP to the parabolic Airy sheet (Theorem~\ref{t.airy sheet}) implies the displayed inequality.
\end{proof}



Now we turn to the proof of Lemma~\ref{l.i j properties} on the ordering and monotonicity properties of $i^\lambda_y$. Based on the intuition in terms of planarity of infinite geodesics explained after Lemma~\ref{l.i j properties}, the proof idea for each property is to show that assuming the contrary leads to a contradiction between the definition of $i$ as the maximizer and the relations imposed by planarity via Corollary~\ref{c.a-b monotonicity}.

\begin{proof}[Proof of Lemma~\ref{l.i j properties}]
We start with (i). Suppose to the contrary that there is an $x^*\geq 0$ such that $i^\lambda_y(x^*) > i^\lambda_{y'}(x^*)$. Let us for shorthand call these values $i$ and $j$, so that we have assumed $i>j$. Then,
\begin{align*}
b^{\lambda,y}_i + \cP^{\lambda}_{i\srt 1}(x^*) &> b^{\lambda,y}_j + \cP^{\lambda}_{j\srt 1}(x^*)\\
b^{\lambda,y'}_j + \cP^{\lambda}_{j\srt 1}(x^*) &\geq b^{\lambda,y'}_i + \cP^{\lambda}_{i\srt 1}(x^*);
\end{align*}
the first is a strict inequality since $i = i^\lambda(x^*)$ is the minimum index which achieves the supremum in \eqref{e.airy sheet boundary LPP problem} and we have assumed $i>j$. The above pair of inequalities implies that 
$\smash{b^{\lambda,y'}_j} - \smash{b^{\lambda,y}_j} > \smash{b^{\lambda,y'}_i} - \smash{b^{\lambda,y}_i}$, which contradicts Corollary~\ref{c.a-b monotonicity}.

Next we turn to (ii). We start with showing that $i^{\lambda}_y$ is non-decreasing.

Again to the contrary, suppose we have $x_1 < x_2$ with $i^{\lambda}_y(x_1) > i^{\lambda}_y(x_2)$. Let $i_1 = i^{\lambda}_y(x_1)$ and $i_2=i^{\lambda}_y(x_2)$ so that $i_1 > i_2$. Then,
\begin{align*}
b^{\lambda, y}_{i_1} + \cP^{\lambda}_{i_1\srt 1}(x_1) &> b^{\lambda, y}_{i_2} + \cP^{\lambda}_{i_2\srt 1}(x_1)\\
b^{\lambda, y}_{i_2} + \cP^{\lambda}_{i_2\srt 1}(x_2) &\geq b^{\lambda, y}_{i_1} + \cP^{\lambda}_{i_1\srt 1}(x_2);
\end{align*}
again the first inequality is strict since $i_1$ is the minimum index which achieves the supremum in \eqref{e.airy sheet boundary LPP problem}. These two inequalities combined imply that 
$$\cP^{\lambda}_{i_1\srt 1}(x_1) + \cP^{\lambda}_{i_2\srt 1}(x_2) > \cP^{\lambda}_{i_1\srt 1}(x_2) + \cP^{\lambda}_{i_2\srt 1}(x_1),$$
which is equivalent to 
\begin{align*}
\MoveEqLeft[16]
\cP[(\lambda, i_1)\to (\lambda+x_1, 1)] + \cP[(\lambda, i_2) \to (\lambda+x,1)]\\
&> \cP[(\lambda, i_1)\to(\lambda+x_2,1)] + \cP[(\lambda,i_2)\to(\lambda+x_1,1)];
\end{align*}
but since $i_1 > i_2$ and $x_1<x_2$, this contradicts Lemma~\ref{l.crossing}.

Now we turn to showing $\smash{i^\lambda_y}$ is left-continuous. Left-continuity is a consequence of the fact that $b_i^{\lambda,y} + \cP_{i\srt 1}^\lambda(x)$ is a continuous function of $x$ for each $i\in\N$ and $i^\lambda_y(x)$ is defined as the minimum index which achieves the supremum in \eqref{e.airy sheet boundary LPP problem}. Indeed, suppose that $\lim_{x\uparrow x^*} i^\lambda_y(x) = i^*$. This implies that 
\begin{align*}
b_{i^*}^{\lambda,y} + \cP^\lambda_{i^*\srt 1}(x) \geq b_{i}^{\lambda,y} + \cP^\lambda_{i\srt 1}(x)
\end{align*}
for all $i\neq i^*$ and $x\in[x^*-\varepsilon, x^*)$ with $\varepsilon>0$ sufficiently small, since, for $x$ in this interval, $i^\lambda(x) = i^*$. Taking the limit of $x\uparrow x^*$ shows that the displayed inequality also holds for $x=x^*$. Since $i^\lambda_y(x^*)$ is the minimum index achieving the supremum in \eqref{e.airy sheet boundary LPP problem}, we have that $i^\lambda_y(x^*) \leq i^*$. But the monotonicity of $i^\lambda_y$ in $x$ implies $i^\lambda_y(x^*)\geq i^*$, which completes the proof of part (ii) of Lemma~\ref{l.i j properties}.
%
\end{proof}

In the rest of the paper we will typically only be working with the starting position $y_b$, and so we will drop the $y$ dependence in the notation for boundary data. That is, $\{b^\lambda_i\}_{i\in\N}$ will be the boundary data from $y_b$.

\subsection{Boundary data process}

In many of our arguments we will need a boundary data \emph{process}, as we explained in the proof overview in Section~\ref{s.intro.proof strategy}. But because the limits \eqref{e.a b definition} defining $\{b_i^\lambda\}_{i\in\N}$ are only known to exist almost surely for fixed $\lambda$, and it is not a priori clear from its definition that it should be continuous in $\lambda$, we need a different specification of the boundary data as a process in the location $\lambda$. This is the role of the next definition, which is based on the idea that we can get boundary data at $\lambda+x$ by considering the boundary data at $\lambda$ and an LPP problem on $[\lambda,\lambda+x]$.

For $\lambda\in\R$ and $i\in\N$, define $Z^{\lambda,b}_i:[0,\infty)\to\R$ by 
\begin{equation}\label{e.Z_i^lambda defn}
Z_i^{\lambda,b}(x) = \sup_{j\geq i}\left\{b_j^\lambda + \cP^\lambda_{j\srt i}(x)\right\} = \sup_{j\geq i}\left\{b_j^\lambda + \cP[(\lambda,j)\to(\lambda+x,i)]\right\}.
\end{equation}
Note that, almost surely, for all $i\in\N$ and $x\geq 0$, $Z^{\lambda,b}_i(x)\geq Z^{\lambda,b}_{i+1}(x)$. We also have that $\smash{\S(y_b, \lambda+x) = Z_1^{\lambda,b}(x)}$ for all $x\geq 0$ by Lemma~\ref{l.Airy sheet to boundary LPP problem}. However, we can get another representation of $\S(y_b, \lambda+ x)$ in terms of all the $Z_i^{\lambda,b}$ analogous to Lemma~\ref{l.Airy sheet to boundary LPP problem}, which is the sense in which $\smash{Z^{\lambda,b}_i(x)}$ is the boundary data from $y_b$ at $x$ on line $i$ and can be thought of as the process $\smash{x\mapsto b_i^{\lambda+x}}$.

\begin{lemma}\label{l.S Z^lambda representation}
Let $\lambda\in\R$ be fixed. Then almost surely, for all $x\geq y\geq 0$,
$$\S(y_b, \lambda+ x) = \sup_{i\in\N} \left\{Z^{\lambda,b}_i(y) + \cP[(\lambda+y,i)\to(\lambda+x,1)]\right\}.$$
Moreover, the supremum is attained at a finite index.
\end{lemma}

\begin{proof}
We have from Lemma~\ref{l.Airy sheet to boundary LPP problem} and the definition \eqref{e.Z_i^lambda defn} of $Z^{\lambda,b}_1$ that
$$\S(y_b, \lambda+x) = Z^{\lambda,b}_1(x).$$
Expanding the definition of the righthand side, and decomposing $\cP[(\lambda, j)\to (\lambda+x,1)]$ based on the location of the corresponding geodesic at $\lambda+y$, we see that
\begin{align*}
Z^{\lambda,b}_1(x) &= \sup_{j\in\N}\left\{b^\lambda_j + \cP[(\lambda,j)\to (\lambda+x,1)]\right\}\\%
&=\sup_{j\in\N} \sup_{i\leq j} \left\{b^\lambda_j + \cP[(\lambda, j)\to (\lambda+y, i)] + \cP[(\lambda+y,i)\to(\lambda+x,1)]\right\}\\
&=\sup_{i\in\N}\left\{Z^{\lambda,b}_i(y) + \cP[(\lambda+y,i)\to(\lambda+x,1)]\right\};
\end{align*}
for the final equality, we use that $\sup_{j\in\N} \sup_{i\leq j} a_{ij} = \sup_{i\in\N} \sup_{j\geq i} a_{ij}$ for any real collection $\{a_{ij}\}$, and the definition \eqref{e.Z_i^lambda defn} of $\smash{Z^{\lambda,b}_i}$. Observe from this manipulation that if the supremum in the first equality is attained at a finite index, then the same is true of the supremum in the final equality. Since we know from Lemma~\ref{l.Airy sheet to boundary LPP problem} that the first supremum is indeed attained at a finite index, the proof of Lemma~\ref{l.S Z^lambda representation} is complete.
\end{proof}


\section{Absolute continuity to Brownian motion and its running maximum}
\label{s.abs cont to br local time}

In this section we develop two main absolute continuity statements. The first compares $Z^{\lambda,b}_i$ to Brownian motion. The second is slightly more technical, and compares the running maximum of a process which is Brownian-like ``away from zero'' to the running maximum of Brownian motion. As indicated in Section~\ref{s.intro.proof strategy}, these will be needed in the proof of Theorem~\ref{t.patchwork}.

We will be making several absolute continuity statements in this section and the rest of the paper, and we introduce some convenient notation and terminology to streamline this. 
For two probability measures $\mu_1$ and $\mu_2$ on a measure space, $\mu_1\ll \mu_2$ will denote that $\mu_1$ is absolutely continuous with respect to $\mu_2$. We will often abuse notation and, for random variables $X_1$ and $X_2$, say that $X_1\ll X_2$ to mean that the distributions of $X_1$ and $X_2$ satisfy the corresponding relation; obviously, the joint distribution of $X_1$ and $X_2$ has no relevance to the statement. For two processes $X$ and $Y$ defined on $[0, \infty)$, if $X|_{[\varepsilon,T]}\ll Y|_{[\varepsilon, T]}$ for every $0<\varepsilon<T$, we will sometimes say that $X$ is \emph{locally absolutely continuous away from zero} to $Y$.

Fixing $\lambda\in\R$, recall the definition of $\smash{\cP^\lambda_i}$ from \eqref{e.R definition}. In Section~\ref{s.proofs}, to prove Theorem~\ref{t.patchwork} on the absolute continuity comparison of $\wdf$ and $\cL$, we will observe that $\wdf$ can be essentially written (i.e., if $\lambda$ is taken sufficiently negative, so that the infinite geodesic to $(x,1)$ is at line two or lower at $\lambda$), for $x\geq 0$, as
$$\wdf(\lambda+x) = \max_{0\leq s\leq x} \left(Z^{\lambda,b}_2(s) - \cP^\lambda_1(s)\right)$$
(see Section~\ref{s.intro.proof strategy} for an indication of why this would be true). The comparison of $\wdf$ to $\cL$ could be made if we knew that $\smash{(\cP^\lambda_1, Z^{\lambda,b}_2)}$ is absolutely continuous to independent Brownian motions, which would imply that the running maximum of $\smash{Z^{\lambda,b}_2 - \cP^\lambda_1}$ is absolutely continuous to the running maximum of Brownian motion (which has the same distribution as $\cL$ by L\'evy's identity, recalled ahead as Proposition~\ref{p.levy identity}). However, we will in fact only have the absolute continuity of $\smash{(\cP^\lambda_1, Z^{\lambda,b}_2)}$ to independent  Brownian motions away from zero, i.e., on $[\varepsilon, T]$ for every $\varepsilon>0$. But we have control over $\lambda$ and can shift it back so that the domain in which we are interested is indeed away from $\lambda$, in the sense that the maximum in the previous display is attained after $\varepsilon$ for every $x$ in a fixed interval with high probability; thus it is sufficient for us to prove an absolute continuity statement for the increment of the running maximum of a process $X$ under the hypothesis that $X$ is absolutely continuous to Brownian motion on its whole domain, which is straightforward.


We now give the precise versions of the two statements we described, and their proof is the main goal of this section. Even though our immediate need concerns only the absolute continuity of $(\cP^\lambda_1, Z^{\lambda,b}_2)$ to Brownian motion, the first result is stated more generally, and will also be of use in the general form for later arguments.


\begin{proposition}\label{p.Br abs cont of Z}
Let $j\in \N$ and $\lambda\in\R$. For any $0<\varepsilon<T$, we have that $(\cP^\lambda_1, \ldots, \cP^\lambda_{j-1}, Z^{\lambda, b}_j)|_{[\varepsilon,T]} \ll (B_1, \ldots, B_j)|_{[\varepsilon,T]}$, where $B_1, \ldots, B_j$ are independent rate two Brownian motions.
\end{proposition}


\begin{lemma}\label{l.abs cont of max increment}
Let $X:[0,\infty)\to \R$ be a process such that $X(\cdot)-X(0)|_{[0,T]} \ll B|_{[0,T]}$ for all $T>0$, where $B$ is a Brownian motion of rate $\sigma^2$ on $[0,\infty)$ started at zero. Let $M^{X, \mathrm{inc}}(t) = \max_{0\leq s\leq t} X(s) - X(0)$ (where $\mathrm{inc}$ indicates that it is the running maximum of the increment), and $M^B$ be defined analogously. Then for any $[c,d]\subset [0,\infty)$, $M^{X,\mathrm{inc}}(\cdot) - M^{X,\mathrm{inc}}(c)|_{[c,d]} \ll M^{B}(\cdot)- M^B(c)|_{[c,d]}$.
\end{lemma}

While $M^{X,\mathrm{inc}}(\cdot) - M^{X,\mathrm{inc}}(c)$
 no longer involves $X(0)$, we include it in the definition of $M^{X,\mathrm{inc}}(\cdot)$ because the hypothesis is stated in terms of $X(\cdot)- X(0)$; indeed, for our application, the input absolute continuity to Brownian motion will be true only after subtracting the initial value.




We prove Proposition~\ref{p.Br abs cont of Z} in Section~\ref{s.abs cont of boundary data process} and Lemma~\ref{l.abs cont of max increment} in Section~\ref{s.pitman trans abs cont}.

\subsection{Brownian absolute continuity of boundary data processes}\label{s.abs cont of boundary data process}

We will make use of the following statement, which was proved in \cite{sarkar2020brownian}, to prove Proposition~\ref{p.Br abs cont of Z}.

\begin{proposition}[Theorem~4.3 of \cite{sarkar2020brownian}]\label{p.Br abs cont of LPP with boundary data}
Let $n\in\N$ and $B = (B_1, \ldots, B_n)$ be a collection of $n$ independent rate two Brownian motions. Let $b=(b_1, \ldots,  b_n)\in\R^n$ be a deterministic vector. Define
$$\mc H_{n, b}(x) = \max_{1\leq i\leq n}\Bigl\{b_i + B[(0,i)\to(x,1)] \Bigr\}.$$
Then $\mc H_{n,b}|_{[\varepsilon,T]} \ll B'|_{[\varepsilon,T]}$, where $0<\varepsilon<T$ and $B'$ is a rate two Brownian motion.
\end{proposition}

The proof of Proposition~\ref{p.Br abs cont of LPP with boundary data} in \cite{sarkar2020brownian} goes essentially via first proving the absolute continuity away from zero to Brownian motion of the \emph{Pitman transform} of two functions which are themselves absolutely continuous away from zero to Brownian motion; in the case that the functions are exactly Brownian motion, the Pitman transform is the $n=2$ and $b=0$ case of $\mc H_{n,b}$. Then on iterating the argument and making use of a composition property, the general $n$ case is obtained. We will introduce a variant of their Pitman transform in Section~\ref{s.skorohod representation}.





To verify the absolute continuity required to apply Proposition~\ref{p.Br abs cont of LPP with boundary data} to prove Proposition~\ref{p.Br abs cont of Z}, we would need to know that the index at which the supremum in the definition \eqref{e.Z_i^lambda defn} of $\smash{Z^{\lambda,b}_i}$ is achieved is almost surely finite. This is the content of the following proposition.

\begin{proposition}\label{p.maximizer index is finite for general i}
Almost surely, for all $x\geq 0$ and $j\in\N$, the supremum in the definition \eqref{e.Z_i^lambda defn} of $\smash{Z^{\lambda, b}_j(x)}$ is attained at a finite index. The index is non-decreasing in $x$; in particular, it is uniformly bounded for $x$ in any given compact set.
\end{proposition}

While we prove Proposition~\ref{p.maximizer index is finite for general i} in Section~\ref{s.Z^lambda index finite}, we first use it to give the proof of Proposition~\ref{p.Br abs cont of Z}.

\begin{proof}[Proof of Proposition~\ref{p.Br abs cont of Z}]
Proposition~\ref{p.maximizer index is finite for general i} implies that the index achieving the supremum in the definition \eqref{e.Z_i^lambda defn} of $\smash{Z^{\lambda, b}_i(x)}$ is at most a random finite constant $K$ for all $x\in[\varepsilon/2,T]$. On $\{K=k\}$, $\smash{Z^{\lambda, b}_j}$ is a function of $\smash{\cP^\lambda_j, \ldots, \cP^\lambda_k}$: it is equal to the function 
$$x\mapsto \max_{j\leq i\leq k}\Bigl\{ b^\lambda_i + \cP[(\lambda, i) \to (x,j)]\Bigr\}.$$
Condition on the $\sigma$-algebra generated by $\{\cP_i(x): i\in\N, x\leq \lambda\}$. By the Brownian Gibbs property, under this conditioning, $(\cP^\lambda_1, \ldots, \cP^\lambda_k)$ is absolutely continuous to $k$ independent rate two Brownian motions on $[0, T]$. Now, define
$$\mc H_{j,k, b^\lambda}(x) := \max_{j\leq i\leq k}\Bigl\{ b^\lambda_i + B[(\lambda, i) \to (x,j)]\Bigr\},$$
where $B$ is a collection of independent rate two Brownian motions.

Under the above conditioning, for any $k\in\N$, on the event $\{K=k\}$, we have 
$$(\cP^\lambda_1, \ldots, \cP^\lambda_{j-1}, Z^{\lambda,b}_{j}) \ll (B_1, \ldots, B_{j-1}, \mc H_{j,k, b^\lambda}).$$
Since $\mc H_{j,k, b^\lambda}$ is independent of $(B_1, \ldots, B_{j-1})$, and by Proposition~\ref{p.Br abs cont of LPP with boundary data} we have that $\mc H_{j,k,b^\lambda}$ is absolutely continuous away from zero to Brownian motion. Removing the conditioning and doing a union bound over $k\in\N$ completes the proof of Proposition~\ref{p.Br abs cont of Z}.
\end{proof}

\subsection{Absolute continuity of running maximum processes to Brownian local time}
\label{s.pitman trans abs cont}

The proof of Lemma~\ref{l.abs cont of max increment} is straightforward and we turn to it now.

\begin{proof}[Proof of Lemma~\ref{l.abs cont of max increment}]
Observe that, for $t\geq c$,
$$M^{X, \mathrm{inc}}(t) - M^{X, \mathrm{inc}}(c) = \max_{0\leq s \leq t} \Bigl(X(s) - X(0)\Bigr) - \max_{0\leq s \leq c} \Bigl(X(s) - X(0)\Bigr)$$
Since the righthand side is a function of $(X(\cdot)-X(0))|_{[0,d]}$, which is absolutely continuous to a Brownian motion on $[0,d]$, we see that the righthand side is absolutely continuous to the process (on $[c,d]$)
$$t\mapsto \max_{0\leq s \leq t} B(s) - \max_{0\leq s \leq c} B(s),$$
as was to be shown.
\end{proof}

\section{Proofs of the absolute continuity and Hausdorff dimension results}\label{s.proofs}

We start by formally defining Brownian local time (at zero), which, heuristically, measures the amount of time Brownian motion spends at the origin.

\begin{definition}[Brownian local time]\label{d.local time}
Let $B:[0,\infty)\to\R$ be a one-dimensional Brownian motion of rate $\sigma^2$ started at the origin. The associated Brownian local time at zero $\cL:[0,\infty)\to [0,\infty)$ of rate $\sigma^2$ is defined as
$$\cL(t) = \lim_{\varepsilon\downarrow 0}\frac{1}{2\varepsilon} \int_0^t \one_{-\varepsilon < B(s) < \varepsilon}\, \mathrm ds.$$
\end{definition}

In our arguments, we will generally relate last passage values to the running maximum process of various processes which are absolutely continuous to Brownian motion. We move from these running maximum processes to the local time process of Brownian motion via the famous identity of L\'evy; see for example \cite[Theorems~6.16 and 7.38]{morters2010brownian}.

\begin{proposition}[L\'evy's identity]\label{p.levy identity}
Let $B:[0,\infty)\to\R$ be Brownian motion of rate $\sigma^2$ started at the origin. Let $\cL:[0,\infty)\to [0,\infty)$ be its associated local time at zero and $M:[0,\infty)\to [0,\infty)$ be its running maximum process, i.e., $M(t) = \sup_{s\leq t} B(s)$. {Then $(\cL, |B|)$ and $(M, M-B)$ are equal in law as processes on $[0,\infty)$.}
\end{proposition}

\subsection{Non-degeneracy of $\NC(\wdf)$}
In the proof of Theorem~\ref{t.patchwork} we will need that the range of $\wdf$ is $\R$, which also implies that $\NC(\wdf)$ (which recall is the set of non-constant points of $\wdf$) is non-empty. These facts were essentially proven in \cite[Proposition~3.10]{basu2019fractal} in a pre-limiting setting and we reproduce the argument for completeness. 

\newcommand{\Sstat}{\S^\cup}

\begin{lemma}\label{l.strong non-degeneracy of E}
With probability one, $\lim_{M\to\infty} \wdf(M) = \infty$ and $\lim_{M\to-\infty} \wdf(M) = -\infty$.
\end{lemma}

\begin{proof}
The proof of the second limit is similar to that of the first, and we prove only the first here.

Let $\Sstat$ be the parabolic Airy sheet with the parabola compensated for, i.e., $\Sstat(y,x) = \S(y,x) + (y-x)^2$. The distribution of $\Sstat(y,x)$ is the same for all $x,y\in\R$, which follows from the stationarity of $\S$, that $\S(0,\cdot) = \cP_1(\cdot)$ from Definition~\ref{d.airy sheet}(ii), and that the Airy$_2$ process (given by $x\mapsto \cP_1(x) + x^2$) is stationary. (In fact, this implies that the common distribution is the GUE Tracy-Widom distribution \cite{tracy1994level}, though we will not make use of this fact.)

It is enough to show that $\limsup_{M\to\infty} \wdf(M) = \infty$ almost surely. We observe that
\begin{equation}\label{e.W not constant}
\P\left(\limsup_{M\to\infty}  \wdf(M) = \infty\right) \geq \limsup_{M\to\infty} \P\bigl(\wdf(M) > 2M^{1/2}\bigr).
\end{equation}
Writing out $\wdf(M)$ we see
\begin{align*}
\wdf(M) = \S(y_b,M) - \S(y_a, M)
&= \Sstat(y_b,M) - \Sstat(y_a, M) - (y_b-M)^2 +(y_a-M)^2\\
&= \Sstat(y_b,M) - \Sstat(y_a, M) - (y_b^2-y_a^2) + 2M(y_b-y_a).
\end{align*}
Note that $\Sstat(y_b,\pm M)$, $\Sstat(y_a,\pm M)$ all have the same distribution and are almost surely finite. For notational simplicity, let $X$ be a random variable with the same distribution. Now, by a union bound,
\begin{align*}
\P\bigl(\wdf(M) \leq 2M^{1/2}\bigr) \leq \P\bigl(X\leq M^{1/2} + y_a^2 - (y_b-y_a)M\bigr) + \P\bigl(X\geq -M^{1/2} - y_b^2 + (y_b-y_a)M\bigr).
\end{align*}
Since $X$ is almost surely finite and $y_b>y_a$, this implies that 
$$\limsup_{M\to\infty}\P\bigl(\wdf(M) - \wdf(-M) > 2M^{1/2}\bigr) = 1,$$
which, with \eqref{e.W not constant}, completes the proof of Lemma~\ref{l.strong non-degeneracy of E}.
\end{proof}

\subsection{A Pitman transform representation of $Z^{\lambda,b}_i$ and the proof of Theorem~\ref{t.patchwork}}\label{s.skorohod representation}

As we mentioned in Section~\ref{s.abs cont to br local time}, to prove Theorem~\ref{t.patchwork} we will write $\wdf(x)$ as 
$$\wdf(x) = \max_{\lambda \leq s\leq x}\left(Z^{\lambda, b}_2(s-\lambda) - \cP_1(s)\right)$$
under the simplifying assumption that the infinite geodesic ending at $(x,1)$ enters the top line only after $\lambda$, which will hold for all large enough $x$; in the actual formula (see Lemma~\ref{l.Z_i^lambda recursion}) there will be an extra term for the case that the geodesic is on the top line at $\lambda$.
Since $\S(y_a,x) = \S(0,x) = \cP_1(x)$, the previous equation is equivalent to
$$\S(y_b,x) = \cP_1(x) + \max_{\lambda \leq s\leq x}\left(Z^{\lambda, b}_2(s - \lambda) - \cP_1(s)\right).$$
Note that this is the LPP value from $(\lambda, 2)$ to $(x,1)$ in the two-line environment with top line $\cP_1$ and bottom line $\smash{Z^{\lambda,b}_2(\,\cdot\, -\lambda)}$. 

Since $\S(y_b,x) = Z^{\lambda, b}_1(x-\lambda)$ for $x\geq \lambda$ by the definition \eqref{e.Z_i^lambda defn} of $Z^{\lambda,b}_1$, to prove the above displayed formula for $\S(y_b,x)$, we can proceed by finding a representation of $\smash{Z^{\lambda, b}_1}$ in terms of $\smash{Z^{\lambda,b}_2}$ and $\cP_1$.
To explain the representation that we will obtain, recall from \eqref{e.Z_i^lambda defn} that $\smash{Z^{\lambda, b}_1}$ is given by
$$Z_i^{\lambda, b}(x) = \sup_{j\geq i}\left\{b_j^\lambda + \cP[(\lambda, j) \to (\lambda+x, i)]\right\}.$$
Essentially, $Z^{\lambda, b}_1(x)$ is the best weight from $\lambda$ to $(\lambda+x,1)$ with the boundary data $\{b^\lambda_i\}_{i\in\N}$. Thus, heuristically, it should be expressible as a two-line LPP problem with top line $\cP_1$ and bottom line~$\smash{Z^{\lambda, b}_{2}}$, as this should capture the case that the geodesic enters the top line after $\lambda$ (leaving aside for now the case that it enters the top line \emph{before} $\lambda$), as mentioned above.

However, observe that $\smash{Z^{\lambda, b}_{2}(0) = b^\lambda_{2}}$ and not zero. Since the LPP value considers increments, the boundary value $b^\lambda_2$ would be lost if we considered simply the LPP problem.
To handle this, we introduce notation for the LPP value across two lines that includes the boundary data on the second line. For two continuous functions $f_1, f_2 :[0,\infty)\to\R$, we call this the Pitman transform $\PT(f_1,f_2)$, defined by
\begin{equation}\label{e.G' defn}
(\PT(f_1,f_2))(x) = f_1(x) + \max_{0\leq s\leq x}(f_2(s) - f_1(s)).
\end{equation}
This is similar to the definition of the transformation with the same name in \cite{sarkar2020brownian}. We next record a few basic properties of $\PT$.

%

%

\begin{lemma}\label{l.modified gap properties}
Let $a\in\R$. Let $f_1,f_2, f_2^i:[0,\infty)\to \R$ be continuous for $i\in\N$. Then,%
\begin{enumerate}
	\item $a+\PT(f_1, f_2) = \PT(f_1, a+f_2)$, where $(a+f_2)(x) = a+f_2(x)$

	\item $\max_{i\in\N} \PT(f_1,f^i_2) = \PT(f_1, \max_{i\in\N} f^i_2)$.
\end{enumerate}
\end{lemma}
The proof of Lemma~\ref{l.modified gap properties} is trivial and we omit it.

Now we may state the recursive formula for $Z^{\lambda, b}_i$. Essentially, $\smash{Z^{\lambda, b}_i}$ is $\cP_i$ reflected off of $\smash{Z^{\lambda, b}_{i+1}}$ (in the sense of Skorohod reflection; see \cite{warren2007dyson}, mentioned earlier in Section~\ref{s.intro.proof strategy}, for work within KPZ on Brownian motions reflecting off of other Brownian objects, namely Dyson's Brownian motion).

\begin{lemma}\label{l.Z_i^lambda recursion}
Let $\lambda\in\R$ and recall that $\cP_j^\lambda(x) = \cP_j(\lambda+x) - \cP_j(\lambda)$ for all $j\in\N$.  Almost surely, for all $i\in\N$ and $x\geq 0$,
$$Z^{\lambda, b}_i(x) = \max\left\{(\PT(\cP_i^\lambda, Z^{\lambda, b}_{i+1}))(x), b^\lambda_i + \cP^\lambda_i(x)\right\}.$$
\end{lemma}

Essentially, the second term $b_i^\lambda + \cP^\lambda_i(x)$ handles the case that the best path to $(\lambda+x,i)$ comes through the $i$\textsuperscript{th} line at $\lambda$. The first term, which we discussed previously, is the case where the best path to $(\lambda+x,i)$ jumps to the $i$\textsuperscript{th} line after $\lambda$. We will be making use of only the $i=1$ case in our proof of Theorem~\ref{t.patchwork}.

\begin{proof}[Proof of Lemma~\ref{l.Z_i^lambda recursion}]
From the definition \eqref{e.Z_i^lambda defn} of $Z_i^\lambda$,
\begin{align}\label{e.Z_i^lambda decomposition}
Z_i^\lambda(x) = \max_{j\geq i+1}\left\{b_j^\lambda + \cP[(\lambda, j)\to (\lambda+x,i)]\right\} \vee (b_i^\lambda + \cP^\lambda_i(x)). 
\end{align}
So it is enough to show that the first maximization is $\PT(\cP_i^\lambda, Z_{i+1}^\lambda)$. For this, observe that
\begin{align*}
\cP[(\lambda, j)\to(\lambda+x, i)] &= \cP_i(\lambda+x) + \max_{0\leq z\leq x}\big(\cP[(\lambda,j) \to (\lambda+z,i+1)] - \cP_i(\lambda+z)\big)\\
&= \Big(\PT\big(\cP^\lambda_i, \cP[(\lambda, j)\to (\lambda+\bigcdot\,, i+1)]\big)\Big)(x).
\end{align*}
Putting this into \eqref{e.Z_i^lambda decomposition}, applying Lemma~\ref{l.modified gap properties}, and recalling the definition \eqref{e.Z_i^lambda defn} of $Z_{i+1}^\lambda$ completes the proof of Lemma~\ref{l.Z_i^lambda recursion}.
\end{proof}

With these preliminaries, we may give the proof of Theorem~\ref{t.patchwork}. The idea is to write $\S(y_b,x)$ as the running maximum of $\smash{Z^{\lambda,b}_2 - \cP_1}$, and use that the latter process is absolutely continuous away from zero to rate four Brownian motion by Proposition~\ref{p.Br abs cont of Z}. Finally, the increment of the running maximum of such a process is absolutely continuous to Brownian local time by Lemma~\ref{l.abs cont of max increment} and L\'evy's identity.

\begin{proof}[Proof of Theorem~\ref{t.patchwork}]
Recall that we have assumed (without loss of generality) that $y_a = 0$. From Definition~\ref{d.airy sheet} of $\S$, this gives that
$$\S(y_a, x) = \cP_1(x),$$
as processes in $x$. Further, from the definition \eqref{e.Z_i^lambda defn} of $Z^{\lambda, b}_i$ and Lemma~\ref{l.Z_i^lambda recursion}, for any $\lambda\in\R$ and $x\geq \lambda$,
\begin{align*}
\S(y_b,x) = Z^{\lambda, b}_1(x-\lambda)
&= \PT\bigl(\cP^\lambda_1, Z^{\lambda,b}_2\bigr)(x-\lambda)\vee (b_1^\lambda + \cP^\lambda_1(x-\lambda))\\
&= \sup_{\lambda \leq z\leq x}\left(Z^{\lambda, b}_2(z-\lambda) + \cP_1(x) - \cP_1(z)\right) \vee \left(b^\lambda_1+\cP_1(x) - \cP_1(\lambda)\right),
\end{align*}
where we used the definition \eqref{e.G' defn} of the Pitman transform and that $\cP^\lambda_1(x) = \cP_1(\lambda+x) - \cP_1(\lambda)$. 
The preceding two displays imply that
\begin{equation*}
\begin{split}
\wdf(x) &= \sup_{\lambda \leq z\leq x}\left(Z^{\lambda, b}_2(z-\lambda) - \cP_1(z)\right) \vee \left(b^\lambda_1- \cP_1(\lambda)\right)\\
&= \sup_{\lambda \leq z\leq x}\left(Z^{\lambda, b}_2(z-\lambda) - \cP_1(z)\right) \vee \wdf(\lambda),
\end{split}
\end{equation*}
where in the second equality we have used that $b^\lambda_1 = \S(y_b, \lambda)$ from the definition \eqref{e.a b definition} of $b^{\lambda}_i$. Because we don't have absolute continuity of $\smash{Z^{\lambda,b}_2}$ to Brownian motion on $[0, T]$ and only on $[\varepsilon,T]$, we want to split the range over which the maximum is taken so that we can get a separation of $z$ from $\lambda$ (which we will take to be a unit separation), which we do next. Indeed,
\begin{equation}\label{e.D two line representation}
\begin{split}
\wdf(x) &= \sup_{\lambda+1 \leq z\leq x}\left(Z^{\lambda, b}_2(z-\lambda) - \cP_1(z)\right) \vee \sup_{\lambda \leq z\leq \lambda+1}\left(Z^{\lambda, b}_2(z-\lambda) - \cP_1(z)\right) \vee \wdf(\lambda)\\
&= \sup_{\lambda+1 \leq z\leq x}\left(Z^{\lambda, b}_2(z-\lambda) - \cP_1(z)\right) \vee \wdf(\lambda+1),
\end{split}
\end{equation}
where in the last equality we combined the last two terms and wrote it as $\wdf(\lambda+1)$ using the formula for the latter from the previous display.

Let $[c,d]\subset \R$ be fixed. Recall that we wish to show that $\wdf(\cdot) - \wdf(c)|_{[c,d]}$ is absolutely continuous to $\cL(\cdot) - \cL(1)|_{[1, d-c+1]}$. We know from Proposition~\ref{p.Br abs cont of Z} that the increment of $\smash{Z^{\lambda,b}_2 - \cP^\lambda_1}$ on $[1,T]$ is absolutely continuous to Brownian motion for every $T>0$, and so we will be able to argue from Lemma~\ref{l.abs cont of max increment} that the first term in the previous display is comparable to the running maximum of Brownian motion. To establish that this is sufficient, we need to show that the second term, $\wdf(\lambda+1)$, is not the value of $\wdf(x)$ on the interval of interest with high probability, which we do next.

For any given $\delta>0$, let $\lambda < c-1$ be chosen sufficiently negative that
\begin{equation}\label{e.lambda suff negative condition}
\P\Bigl(\wdf(c) > \wdf(\lambda+1)\Bigr) \geq 1-\delta.
\end{equation}
That this is possible follows from Lemma~\ref{l.strong non-degeneracy of E}, which says that $\wdf(\lambda+1) \to -\infty$ almost surely as $\lambda\to-\infty$. On the high probability event of the previous display, \eqref{e.D two line representation} and the non-decreasing nature of $\wdf$ yield that, for all $x \geq c$,
$$\wdf(x) = \sup_{\lambda+1 \leq z\leq x}\left(Z^{\lambda, b}_2(z-\lambda) - \cP_1(z)\right).$$
Now, we know by Proposition~\ref{p.Br abs cont of Z} that 
$$\Bigl(\cP_1(\cdot) - \cP_1(\lambda+1), Z^{\lambda, b}_2(\cdot-\lambda) - Z^{\lambda, b}_2(1)\Bigr)|_{[\lambda+1, \lambda+T]} \ll (B_1, B_2)|_{[0, T-1]},$$
where $B_1$ and $B_2$ are independent rate two Brownian motions on $[0,\infty)$. In particular, this implies that
$$\left(Z_2^{\lambda,b} (\cdot -\lambda) - \cP_1(\cdot) - (Z_2^{\lambda,b} (1) - \cP_1(\lambda+1))\right)\Bigr|_{[\lambda + 1, \lambda+T]} \ll B|_{[0, T-1]},$$
where $B$ is a rate four Brownian motion. We also know that on the high-probably event mentioned in \eqref{e.lambda suff negative condition}, for $x\in[c,d]$,
$$\wdf(x) - \wdf(c) = \sup_{\lambda+1 \leq z\leq x} \left(Z^{\lambda, b}_2(z-\lambda) - \cP_1(z)\right) - \sup_{\lambda+1\leq z\leq c} \left(Z^{\lambda, b}_2(z-\lambda) - \cP_1(z)\right).$$
Applying Lemma~\ref{l.abs cont of max increment} and L\'evy's identity Proposition~\ref{p.levy identity}, we see that the right hand side of the previous display is absolutely continuous to $\cL(\cdot) - \cL(1)$ as a process on $[1, d-c+1]$.

Thus, for any $A\subseteq \mc C([c,d])$ such that 
$$\P\left((\cL(\cdot) - \cL(1))\bigr|_{[1,d-c+1]} \in A\right) = 0,$$
we have, from the bound \eqref{e.lambda suff negative condition} on $\P(\wdf(c) > \wdf(\lambda+1))$, that
$$\P\left((\wdf(\cdot) - \wdf(c))\bigr|_{[c,d]} \in A\right) \leq 0 + \P\Bigl(\wdf(c) = \wdf(\lambda+1)\Bigr) \leq \delta.$$
Since $\delta>0$ was arbitrary, the proof of Theorem~\ref{t.patchwork} is complete.
\end{proof}

Given Theorem \ref{t.patchwork}, the proof of Corollary~\ref{c.HD is one-half} is quick, and we proceed to it next.

\subsection{The Hausdorff dimension of $E$: proving Corollary~\ref{c.HD is one-half}} \label{s.proof of HD theorem}
We start by recalling the definition of Hausdorff dimension.

\begin{definition}[Hausdorff dimension]\label{d.hd}
The $\alpha$-Hausdorff measure of a set $A\subseteq \R$, denoted $H^\alpha(A)$, is defined as $H^\alpha(A) = \lim_{\delta\downarrow 0} H^\alpha_\delta(A)$, where, for $\delta>0$,
$$H^\alpha_\delta(A) = \inf\left\{\sum_{i=1}^\infty \diam(U_i)^\alpha : A\subseteq \bigcup_{i=1}^\infty U_i \text{ and } 0<\diam(U_i) < \delta\right\}.$$
The Hausdorff dimension of $A$, denoted $\dim(A)$, is defined as
$$\dim(A) = \inf\big\{\alpha>0 : H^\alpha(A) < \infty\big\}.$$
\end{definition}

The Hausdorff dimension enjoys a \emph{countable stability} property; see for example \cite[below Definition~4.8]{mattila1999geometry}. It states that
$$\dim\left(\bigcup_{i\in I} S_i\right) = \sup_{i\in I}\,\dim(S_i),$$
where $I$ is any countable set.

We may now prove Corollary~\ref{c.HD is one-half} using Theorem~\ref{t.patchwork} and the countable stability property of Hausdorff dimension.

\begin{proof}[Proof of Corollary~\ref{c.HD is one-half}]
For $k\in\Z$, let $E_k = \NC(\wdf(\cdot) - \wdf(k))\cap[k, k+1]$. Now clearly,
$$\NC(\wdf) = \bigcup_{k\in\Z} \NC(\wdf)\cap[k,k+1] = \bigcup_{k\in\Z} E_k.$$
The second equality follows by noting that the set of non-constant points of a function is unchanged under constant vertical shifts of the function. 

By the absolute continuity of $\wdf(\cdot)-\wdf(k)$ on $[k,k+1]$ to $\cL(\cdot) - \cL(1)$ on $[1,2]$ and the well-known fact that the set of non-constant points of the latter process (on the event that it is non-empty) has Hausdorff dimension $\frac{1}{2}$ almost surely (see \cite[Theorem~4.24]{morters2010brownian}), it follows that $\dim(E_k) = \frac{1}{2}$ almost surely on the event that $E_k\neq \emptyset$.

Now, since from Lemma~\ref{l.strong non-degeneracy of E} $\NC(\wdf)$ is almost surely non-empty, there must exist at least one $k$ such that $E_k$ is non-empty. By the countable stability property, we obtain that $\dim(\NC(\wdf)) = \frac{1}{2}$ almost surely.
\end{proof}

\begin{remark}It is worth mentioning that besides Hausdorff dimension, there are other natural notions of dimensions that one might also consider, such as the Minkowski dimension and the packing dimension. While they do not agree in general, it is known that, for Brownian local time, all these notions of dimension agree and are equal to $1/2$. This can then be used to obtain analogues of Corollary~\ref{c.HD is one-half} for all the different notions.
\end{remark}

\section{Proof that local limit of $\wdf$ is Brownian local time}\label{s.local limit proof}

\subsection{Outline and setup for the proof} 
We first recall and expand on the proof ideas of Theorem~\ref{t.local limit} as discussed in Section~\ref{s.intro.proof strategy}. As in the latter section, this overview will only discuss the case where $\tau = \tau_\lambda$, as the important conceptual points of the proof are captured here. Recall for $\lambda\in\R$ that $\tau_\lambda$ is the first time after $\lambda$ that $\wdf$ is non-constant. 

Heuristically, the geodesic in $\cP$ corresponding to $\S(y_a, \cdot)$ with $y_a=0$ uses only the top line, since we have $\S(0, x) = \cP_1(x)$. Consider also the geodesic $\pi_b^t$ with boundary data $\{b^\lambda_i\}_{i\in\N}$. When $t=\tau_\lambda$, both geodesics are at the top line at $\tau_\lambda$ since they end at $(\tau_\lambda,1)$, although it may intersect the line $\{\tau_\lambda\}\times \N$ at other levels as well; in fact, as we will show, it will also intersect $(\tau_\lambda,2)$. As $t$ increases, if both the geodesics only use the top line in $[\tau_\lambda, t]$, then $\wdf$ will be constant in that interval by Lemma~\ref{l.S Z^lambda representation}, which allows us to express $\S(y_b, \tau_\lambda+x)$ in terms of the boundary data processes $\smash{Z^{\lambda,b}_i}$ at $\tau_\lambda$. 

Thus by definition of $\tau_\lambda$ (it being a point of increase), at least one of the geodesics must jump to a different line (i.e., its line index at $\tau_\lambda$ increases) as soon as $t$ increases past $\tau_\lambda$. But we know that the geodesic corresponding to $y_a$ uses only the top line, so $\pi_b^t$ must jump (this should also implied by planarity, as $\pi_b^t$ is to the right of the $y_a$ geodesic). 

If we know that it jumps to the \emph{second} line, 
then $\wdf$ in the right-neighbourhood of $\tau_\lambda$ will  essentially be the running maximum of $\cP_2-\cP_1$; this is because the weight of the geodesic which jumps to line two is the LPP value to $t$ across the top two lines $(\cP_1,\cP_2)$, while that of the one which remains on line one is $\cP_1(t)$, and the difference of the two is the running maximum of $\cP_2-\cP_1$, as we observed previously in Section~\ref{s.skorohod representation}. Since $\cP_2-\cP_1$ is absolutely continuous to rate four Brownian motion, even including the origin, this running maximum is absolutely continuous to Brownian local time, again without avoiding the origin. The local limit at zero of the increment of such a process can easily be shown to be Brownian local time by scale invariant considerations.

Thus there are two new aspects to the proof. 
First, we need to know that at $\tau_\lambda$, the geodesic which jumps does so to the second line and not a lower one (Lemma~\ref{l.point of inc rep}, using Proposition~\ref{p.Z_1 Z_3 never coincide}). Second, we need a statement on local limits of processes absolutely continuous to Brownian local time, where we will make use of the added information that the process can be written as a running maximum of a Brownian-like process (Corollary~\ref{c.local limit local time}).

The first aspect of the proof outline, that the geodesic, on jumping after $\tau$, would jump to the second line and not a lower one, is implied if the boundary data at $\tau$ at the third line is strictly smaller than at the first line, i.e., if $\smash{Z^{\lambda, b}_1(\tau-\lambda) > Z^{\lambda,b}_3(\tau-\lambda)}$ by Lemma~\ref{l.S Z^lambda representation} and the continuity of $x\mapsto \cP[(\tau, i) \to(x,1)]${, and is written out in Lemma~\ref{l.point of inc rep}}. That $b^\lambda_1 > b^\lambda_3$ for any fixed $\lambda$ is true almost surely (though we have not yet proved this), but we require such a statement at the random time $\tau -\lambda$. This is the content of the next proposition.

\begin{proposition}\label{p.Z_1 Z_3 never coincide}
Fix $\lambda\in\R$. Almost surely, for all $\tau>0$, it holds that $Z_1^\lambda(\tau) > Z_3^\lambda(\tau)$.
\end{proposition}

The proof of Proposition~\ref{p.Z_1 Z_3 never coincide}, which we give in Section~\ref{s.Z_1 Z_3 never collide}, has several technical details to be handled, but finally relies on the well-known fact that two-dimensional Brownian motion, started at any given point in the plane, almost surely never subsequently hits the origin.

\begin{remark}
In fact, as we indicated in the proof overview in Section~\ref{s.intro.proof strategy}, it is plausible that a proof for Theorem~\ref{t.local limit} can be formulated by directly appealing to the representation of $\wdf$ in terms of $\cP_1$ and $\smash{Z^{\lambda,b}_2}$ and their joint absolute continuity to Brownian motion away from zero, as in the proof of Theorem~\ref{t.patchwork}. However, we have chosen to follow the approach outlined above as it also yields interesting and non-trivial information about the joint distribution of the boundary data processes such as Proposition~\ref{p.Z_1 Z_3 never coincide}.
\end{remark}

The final aspect of the proof outline was that the local limit of the running maximum of a process absolutely continuous to Brownian motion is Brownian local time. This follows from an analogous statement that the local limit of a process absolutely continuous to Brownian motion is itself Brownian motion. We cite this from \cite{dauvergne2020three}. 

\begin{lemma}[Lemma~4.3 of \cite{dauvergne2020three}]\label{l.local limit BM}
Let $B':[0, T]\to \R$ be a process such that $B'|_{[0,T]} \ll B|_{[0,T]}$ for all $T>0$, where $B$ is Brownian motion of rate $\sigma^2$. Then,
$$\lim_{\varepsilon\to 0}\varepsilon^{-1/2} B'(\varepsilon t) = B(t),$$
where the limit is in distribution in the topology of uniform convergence on compact sets.
\end{lemma}

The argument for Lemma~\ref{l.local limit BM} given in \cite{dauvergne2020three} proceeds by contradiction, but is essentially to identify, for a given event $A$ in a convenient measure-determining class, a functional of $t\mapsto \varepsilon^{-1/2}B(\varepsilon t)$ which converges almost surely to $\P(B\in A)$ as $\varepsilon\to 0$. The functional applied to $t\mapsto B'_\varepsilon(t) := \varepsilon^{-1/2}B'(\varepsilon t)$ converges to $\P(\lim_{\varepsilon\to 0} B'_\varepsilon \in A)$, which, by absolute continuity, must be equal to $\P(B\in A)$.

\begin{corollary}\label{c.local limit local time}
Let $B':[0,\infty)\to \R$ be a stochastic process such that $B|_{[0,T]} \ll B|_{[0,T]}$ for any $T>0$, where $B$ is Brownian motion $B$ of rate $\sigma^2$. Let $Y:[0,\infty)\to [0,\infty)$ be defined by $Y(t) = \max_{0\leq s\leq t} B'(s)$. Let $\cL:[0,\infty)$ be the local time at the origin associated to $B$. Then
$$\lim_{\varepsilon\to 0}\varepsilon^{-1/2}Y(\varepsilon t) = \cL(t),$$
where the convergence is in distribution in the topology of uniform convergence on compact sets.
\end{corollary}

\begin{proof}
We see that
$
Y(\varepsilon t) = \max_{0\leq s\leq \varepsilon t} B'(s) = \max_{0\leq s\leq t} B'(\varepsilon s).
$
It is an easy exercise that the map taking a function defined on $[0,\infty)$ to its running maximum on $[0,\infty)$ is continuous in the topology of uniform convergence on compact sets. Then by Lemma~\ref{l.local limit BM} and the continuous mapping theorem,
$$\lim_{\varepsilon\to 0} \varepsilon^{-1/2}Y(\varepsilon t) = \max_{0\leq s\leq t} B(s),$$
where the limit is in distribution in the same topology. L\'evy's identity (Proposition~\ref{p.levy identity}) completes the proof of the corollary.
\end{proof}



\subsection{Brownian Gibbs at nice random times}
In the proof of Theorem~\ref{t.local limit}, we will need to understand the parabolic Airy line ensemble $\cP$ at the random time $\tau_\lambda$. For this we recall from \cite{corwin2014brownian} the concept of a stopping domain and the associated strong Brownian Gibbs property of $\cP$.

\begin{definition}[Stopping domain and strong Brownian Gibbs]
Let $X:\N\to\R$ be an $\N$-indexed collection of continuous curves. Recall the $\sigma$-algebra $\F_{\mathrm{ext}}(k, \ell, r)$ from Definition~\ref{d.bg} generated by the data external to the top $k$ curves of $X$ on $[\ell, r]$. A pair of random variables $\mf l, \mf r$ is a \emph{stopping domain} for $X_1, \ldots, X_k$ if, for all $\ell < r$,
$$\{\mf l \leq \ell, \mf r \geq r\} \in \F_{\mathrm{ext}}(k,\ell, r).$$
Define the $\sigma$-algebra $\F_{\mathrm{ext}}(k,\mf l, \mf r)$ to be the one generated by events $A$ such that $A\cap \{\mf l\leq \ell, \mf r\geq r\} \in \F_{\mathrm{ext}}(k,\ell, r)$ for all $\ell < r$.

An infinite collection of random continuous curves $X$ has the \emph{strong Brownian Gibbs} property if, for any $k\in\N$ and conditionally on $\F_{\mathrm{ext}}(k,\mf l,\mf r)$, the distribution of $X$ on $\intint{1,  k}\times[\mf l, \mf r]$ is that of $k$ independent Brownian bridges $(B_1, \ldots, B_k)$ of rate \emph{two}, with $B_i(x) = X_i(x)$ for $x\in\{\mf l, \mf r\}$ and $i\in\intint{1, k}$, conditioned on not intersecting each other or $X_{k+1}(\cdot)$ on $[\mf l, \mf r]$.
\end{definition}

\begin{proposition}[Lemma 2.5 of \cite{corwin2014brownian}]\label{p.strong bg}
The parabolic Airy line ensemble $\cP$ has the strong Brownian Gibbs property.
\end{proposition}

\begin{remark}
The definition of a stopping domain and the strong Brownian Gibbs property in \cite{corwin2014brownian} considers ensembles with $N$ curves instead of infinitely many. We require the infinite case as our random time $\tau_\lambda$ is defined via all the curves of $\cP$. However, the proof given in \cite{corwin2014brownian} applies verbatim for the infinite case. 

{Alternatively, the $N=\infty$ case can be derived from the finite $N$ one with a certain augmentation to handle boundary data. Given a stopping domain $[\mf l, \mf r]$ defined in terms of the entire line ensemble, one can apply, for any $\delta>0$, the Brownian Gibbs property on the top $k$ curves on an interval $[-M, M]$ with $M$ chosen such that $[\mf l, \mf r]\subseteq [-M,M]$ with probability at least $1-\delta$. Conditionally on $\F_{\mathrm{ext}}(k, -M, M)$, $[\mf l, \mf r]$ is determined by the conditioned data and $\cP|_{[-M,M]}, \ldots, \cP_k|_{[-M,M]}$. Thus we may consider the $\F_{\mathrm{ext}}(k,-M,M)$-conditional ensemble $\cP|_{[-M,M]}, \ldots, \cP_k|_{[-M,M]}$. Note that this ensemble has the Brownian Gibbs property with the boundary condition $\cP_{k+1}$, i.e., the extra condition that $\cP_k|_{[-M,M]}$ remains above $\cP_{k+1}$.

Thus the finite $N$ case (with $N=k$) with a boundary condition given by the curve $\cP_{k+1}$, for which the argument in \cite{corwin2014brownian} goes through verbatim, can be applied to this ensemble defined on $\intint{1,k}\times[-M,M]$ to obtain the conclusion that $\cP_1|_{[\mf l, \mf r]}, \ldots, \cP_k|_{[\mf l, \mf r]}$ is distributed as non-intersecting Brownian bridges with appropriate endpoints, on the high-probability event that $[\mf l,\mf r]\subseteq [-M,M]$. Taking $\delta\to 0$ and $M\to\infty$ appropriately gives the conclusion.}
\end{remark}

\begin{remark}\label{r.right continuity}
In the proof of Theorem~\ref{t.local limit} we will want to apply the strong Brownian Gibbs property to an interval like $[\tau_\lambda, \tau_\lambda+1]$. But observe that this interval is actually not a stopping domain, for the same reason that $\tau_\lambda$ is not a stopping time with respect to the canonical filtration $\sigma(\cP_i(s) : s\leq t, i\in\N)$ of $\cP$: determining the occurrence of $\tau_\lambda$ requires an ``infinitesimal peak'' into the future.

To address this, we will actually need the strong Brownian Gibbs property to hold with respect to the analogue of right-continuous filtrations in this spatial setting. Indeed, for $\ell<r$ and $k\in\N$, consider the $\sigma$-algebra $\F_{\mathrm{ext}}(k, \ell^+, r^-)$ defined as
$$\F_{\mathrm{ext}}(k, \ell^+, r^-) = \bigcap_{n=1}^\infty \F_{\mathrm{ext}}(k, \ell+n^{-1}, r-n^{-1}).$$
To know that an ensemble $X$ has the strong Brownian Gibbs property with respect to the above family of augmented $\sigma$-algebras, it is sufficient to know that $X$ has the Brownian Gibbs property with respect to the same family. Indeed, the proof of Proposition~\ref{p.strong bg} as given in \cite{corwin2014brownian} goes through verbatim with $\F_{\mathrm{ext}}(k,\ell^+, r^-)$ in place of $\F_{\mathrm{ext}}(k,\ell, r)$ if it is known that $X$ has the Brownian Gibbs property with respect to the former family of $\sigma$-algebras.

To prove that $X$ having the Brownian Gibbs property with respect to $\F_{\mathrm{ext}}(k,\ell, r)$ implies the same with respect to $\F_{\mathrm{ext}}(k,\ell^+, r^-)$, it is sufficient to show a form of Blumenthal's zero-one law for the family of augmented $\sigma$-algebras, i.e., that almost surely for any $A\in \F_{\mathrm{ext}}(k, \ell^+, r^-)$, it holds that
\begin{equation}\label{e.spatial blumenthal}
\P\left(A  \mid \F_{\mathrm{ext}}(k, \ell, r)\right) \in \{0,1\}.
\end{equation}
(The argument to show $X$ has the Brownian Gibbs property with respect to $\F_{\mathrm{ext}}(k, \ell^+, r^-)$ via \eqref{e.spatial blumenthal} is analogous to the one showing that a process which is Markov with respect to its canonical filtration is also so with respect to the right-continuous filtration; see eg. \cite[Proposition~2.14]{revuz2013continuous}.)

That \eqref{e.spatial blumenthal} is true is easy to see. We may assume that $A$ is a function of only $\cP_1|_{[\ell, r]}, \dots, \cP_k|_{[\ell, r]}$, as everything else will be conditioned on. Conditionally on $\F_{\mathrm{ext}}(k, \ell, r)$, the distribution of $\cP_1, \ldots, \cP_k$ on $[\ell, r]$ is that of $k$ independent Brownian bridges with boundary values determined by $\F_{\mathrm{ext}}(k, \ell^+, r^-)$, conditioned not to intersect one another or $\cP_{k+1}$. Since the boundary values are almost surely separated, the Brownian bridges are conditioned on a positive probability event. Further, $A$ lies in the $\sigma$-algebra generated by the germ $\sigma$-algebras of the Brownian bridges as well as their time reversals (so as to capture behaviour near $r$). We may assume that it lies in exactly one of these $\sigma$-algebras. Then by Blumenthal's zero-one law applied to these Brownian bridges (or  their time reversals), \eqref{e.spatial blumenthal} follows.
\end{remark}

\subsection{The local limit argument: proving Theorem~\ref{t.local limit}}
{We begin by proving a representation for $\wdf(\tau + \varepsilon t) - \wdf(\tau)$ at a point of increase $\tau$, using Proposition~\ref{p.Z_1 Z_3 never coincide} which states that the boundary data $Z^{\lambda, b}_3$ of the third curve never equals that of the first curve, $Z^{\lambda, b}_1$.

\begin{lemma}\label{l.point of inc rep}
Let $\tau$ be a random variable which is almost surely a point of increase of $\wdf$, and let $\mc K\subseteq [0,\infty)$ be a compact set. There exists random $\varepsilon_0 = \varepsilon_0(\tau, \mc K)>0$ such that, almost surely, for $0<\varepsilon<\varepsilon_0$ and $t\in\mc K$, 
\begin{equation}\label{e.point of inc representation}
\wdf(\tau+\varepsilon t) - \wdf(\tau) = \max_{0\leq s\leq t}\Bigl[\bigl(\cP_2(\tau+\varepsilon s) - \cP_2(\tau)) - (\cP_1(\tau+\varepsilon s) - \cP_1(\tau)\bigr)\Bigr].
\end{equation}
\end{lemma}

\begin{proof}
Fix $\lambda\in \R$. We will show the representation \eqref{e.point of inc representation} on the event that $\tau > \lambda$, which will suffice as we may take $\lambda\to-\infty$ at the end.

Let $t\in\mc K$. Using Lemma~\ref{l.S Z^lambda representation} with $y=\tau$ and $x=\tau+\varepsilon t$, we see
\begin{align*}
\wdf(\tau+\varepsilon t) &= \sup_{i\in\N} \left\{Z^{\lambda, b}_i(\tau-\lambda) + \cP[(\tau, i)\to (\tau+\varepsilon t, 1)]\right\} -  \cP_1(\tau+\varepsilon t).
\end{align*}
Note that, by Proposition~\ref{p.Z_1 Z_3 never coincide}, $Z^{\lambda,b}_1(\tau-\lambda) > Z^{\lambda,b}_3(\tau-\lambda)$. Since by Lemma~\ref{l.S Z^lambda representation} the supremum in the last displayed equation is attained at finite indices, the continuity of $t\mapsto \cP[(\tau, i)\to (\tau+\varepsilon t, 1)]$ for each $i$ implies that there exists random $\varepsilon_0>0$ (depending on the supremum-achieving indices and on the compact set $\mc K$) such that, for all $0<\varepsilon<\varepsilon_0$ and $t\in\mc K$,
\begin{align*}
\wdf(\tau+\varepsilon t) &= \max_{i=1,2} \left\{Z^{\lambda, b}_i(\tau-\lambda) + \cP[(\tau, i)\to (\tau+\varepsilon t, 1)]\right\} -  \cP_1(\tau+\varepsilon t)
\end{align*}
Now by definition $\wdf(\tau) = Z^{\lambda,b}_1(\tau-\lambda)- \cP_1(\tau)$. We also know, since $\tau$ is a point of increase of $\wdf$ and by the monotonicity (in $t$) of the maximizing index as proven in Lemma~\ref{l.i j properties}, that the maximum in the previous display is attained at $i=2$ for all $0<\varepsilon<\varepsilon_0$ and $t\in\mc K$; this implies by continuity that $\smash{Z^{\lambda, b}_1(\tau-\lambda) = Z^{\lambda, b}_2(\tau-\lambda)}$. It also yields Lemma~\ref{l.point of inc rep} since
\begin{align*}
\wdf(\tau+\varepsilon t) - \wdf(\tau) &= \cP[(\tau,2)\to (\tau+\varepsilon t,1)] - \cP[(\tau,1)\to (\tau+\varepsilon t,1)]\\
&= \max_{0\leq s\leq t}\Bigl(\bigl(\cP_2(\tau+\varepsilon s) - \cP_2(\tau)) - (\cP_1(\tau+\varepsilon s) - \cP_1(\tau)\bigr)\Bigr).
\qedhere
\end{align*}

\end{proof}

\begin{proof}[Proof of Theorem~\ref{t.local limit}]
We must prove that $\varepsilon^{-1/2} (\wdf(\tau + \varepsilon t) - \wdf(\tau))$ converges weakly, as a process on any compact set, to $\cL$, for $\tau$ equal to $\tau_\lambda$, $\smash{\rho^h}$, or $\xi_{[c,d]}$ (the last conditionally on the event that $\NC(\wdf)\cap[c,d]\neq \emptyset$. We start by fixing a compact set $\mc K$.


We will prove in Lemma~\ref{l.random times are points of increase} ahead that all of the types of random variables listed are almost surely points of increase of $\wdf$. Taking this statement for granted, by Lemma~\ref{l.point of inc rep} we have, for each of the three cases of the definition of $\tau$,
\begin{equation}\label{e.local limit prelimit}
\varepsilon^{-1/2}\left(\wdf(\tau+\varepsilon t) - \wdf(\tau)\right) = \varepsilon^{-1/2}\max_{0\leq s\leq t}\Bigl[\bigl(\cP_2(\tau+\varepsilon s) - \cP_2(\tau)) - (\cP_1(\tau+\varepsilon s) - \cP_1(\tau)\bigr)\Bigr].
\end{equation}
We \emph{claim} that the expression being maximized on the righthand side is absolutely continuous, as a process in $s$, to rate four Brownian motion, for each case of $\tau$. Given the claim, Theorem~\ref{t.local limit} follows immediately by applying Corollary~\ref{c.local limit local time}.

It remains to prove the claim for each case of $\tau$.

Now, $\tau_\lambda$, $\rho^h$, and $\rho_c^h$ are each stopping times with respect to the augmented filtration generated by $(\cP_i)_{i\in\N}$ as defined in Remark~\ref{r.right continuity}. Thus it follows that $[\tau, \tau+2]$ is a stopping domain in these two cases. Then, by the strong Brownian Gibbs property and Remark~\ref{r.right continuity}, an argument as in Corollary~\ref{c.abs cont} implies that $\cP_i(\tau+\varepsilon s) - \cP_i(\tau)$ for $i=1$ and $2$ are jointly absolutely continuous to two independent rate two Brownian motions. This implies the claim in these two cases.

So it only remains to prove the claim for $\xi_{[c,d]}$, on the event that $\NC(\wdf)\cap[c,d]\neq\emptyset$. We start by observing a convenient representation for $\xi_{[c,d]}$. Namely, let $U$ be an uniform random variable on the interval $[0, \wdf(d)-\wdf(c)]$, which is conditionally independent of $\wdf$ given $\wdf(d)-\wdf(c)$, and let 
$$\rho^U_c = \inf\bigl\{t > c : \wdf(t) = \wdf(c) + U\bigr\};$$
the infimum is over a non-empty set on the event that $\wdf(d) > \wdf(c)$, i.e., $\NC(\wdf)\cap[c,d] \neq \emptyset$. It is easy to see that $\rho^U_c$ has the same distribution as $\xi_{[c,d]}$ by the definition of the distribution function of the latter. {Also recall that $\smash{\rho^h_c}$ is defined analogously to $\smash{\rho^U_c}$ with $h\geq 0$ in place of $U$, and that $\rho^h_c$ is a stopping time with respect to $\cP$.}

Let $\cP'_U(s) = (\cP_1(\rho^U_c + s) - \cP_1(\rho^U_c), \cP_2(\rho^U_c + s) - \cP_2(\rho^U_c))$ for notational convenience, and $\cP'_h$ be defined with $\rho^h_c$ in place of $\rho^U_c$. We must show that $\cP'_U \ll (B_1, B_2)$, with $B_1$ and $B_2$ independent rate two Brownian motions. So let $A\subseteq \mc C([0,1])$ be an event such that $\P( (B_1, B_2) \in A) = 0$. Now,
\begin{align*}
\P\left(\cP'_U\in A, \wdf(d)-\wdf(c)>0\right) &= \E\left[\P\left(\cP'_U\in A\mid U, \wdf(d) - \wdf(c)\right)\one_{\wdf(d)-\wdf(c)>0}\right] \\
&= \E\left[\int_0^{\wdf(d) - \wdf(c)}\frac{\P\left(\cP'_h\in A\mid \wdf(d) - \wdf(c)\right)}{\wdf(d)-\wdf(c)}\one_{\wdf(d)-\wdf(c) > 0}\, dh\right]\\
&\leq \E\left[\int_0^{\infty}\sum_{k=1}^{\infty}\frac{\P\left(\cP'_h\in A\mid \wdf(d) - \wdf(c)\right)}{\wdf(d)-\wdf(c)}\one_{\wdf(d)-\wdf(c) \in[(k+1)^{-1}, k^{-1}]}\, \mathrm dh\right]\\
&\leq \int_0^{\infty}\sum_{k=1}^{\infty}(k+1)\P\left(\cP'_h\in A\right)\, dh = 0.
\end{align*}
In the final line, we used Fubini's theorem and, in the final equality, that $\cP'_h \ll (B_1,B_2)$ (and so $\P(\cP'_h\in A) = 0$) as above since $\rho^h_c$ is a stopping time. This along with the fact that  $\wdf(d)-\wdf(c)>0$ with positive probability which is proved next, completes the proof of the claim in the case that $\tau = \xi_{[c,d]}$ and the proof of Theorem~\ref{t.local limit}.
\end{proof}
} 

\begin{remark}
The above proof in the case of $\xi_{[c,d]}$ can be easily adapted to prove that the local limit of Brownian local time $\cL$ at a point sampled according to $\cL$ on a compact interval is also $\cL$, a statement for which we were unable to find a reference in the literature.
\end{remark}

{
\begin{lemma}[Positive probability of $\NC(\wdf)\cap{[c,d]}\neq\emptyset$]\label{r.positive prob of non-empty}
Fix $d>c$. Then $\wdf(d)-\wdf(c)>0$ with positive probability.
\end{lemma}

\begin{proof}
As in the proof of Lemma~\ref{l.strong non-degeneracy of E}, we can write, where $\S^{\cup}(y,x) = \S(y,x) + (y-x)^2$,
\begin{align*}
\wdf(d)-\wdf(c) &= \S^\cup(y_b, d) - \S^\cup(y_a,d) - (\S^{\cup}(y_b,c) - \S^{\cup}(y_a,c))\\
&\qquad - (y_b-d)^2 + (y_b-c)^2 + (y_a-d)^2 - (y_a-c)^2.
\end{align*}
Since $\S^{\cup}(y,x)$ has the same distribution for every fixed $y,x\in\R$, taking expectations shows that $\E[\wdf(d)-\wdf(c)] > 0$. That $\wdf(d)-\wdf(c)\geq 0$ almost surely completes the proof.
\end{proof}

We next prove the small missing step of the proof of Theorem~\ref{t.local limit}, that the random locations in question are almost surely points of increase of $\wdf$, before turning to the larger task of proving Proposition~\ref{p.Z_1 Z_3 never coincide}.
}

{
\begin{lemma}\label{l.random times are points of increase}
Each of $\tau_\lambda$, $\rho^h$, $\rho_c^h$, and $\xi_{[c,d]}$ is almost surely a point of increase (the final conditionally on $\NC(\wdf)\cap[c,d]\neq\emptyset$).
\end{lemma}

\begin{proof}
First, $\tau_\lambda$ is a point of increase by definition. 

Next we consider $\rho^h$. Observe that, on the event that $\rho^h$ is not a point of increase, there must be a non-trivial interval with left endpoint $\rho^h$ where $\wdf$ is flat, i.e., equals $h$. In particular, the mentioned event implies the existence of a $q\in\Q$ such that $\wdf(q) = h$. It is thus sufficient, by a union bound over rationals, to show that $\wdf(x)$ is a continuous random variable for any fixed $x\in\R$.

We will use the symmetry property of the parabolic Airy sheet $\S$ that $\S(y,x) \stackrel{d}{=} \S(-x,-y)$ as processes on $\R^2$ (see \cite[Lemma~9.1]{dauvergne2018directed}; this follows by a similar relation that holds in Brownian LPP) and the stationarity of $\S$ from item (i) of Definition~\ref{d.airy sheet}. Now,
\begin{align*}
\wdf(x) = \S(y_b,x) - \S(y_a,x) &\stackrel{d}{=} \S(-x,-y_b) - \S(-x,-y_a)\\
&\stackrel{d}{=} \S(0, x-y_b) - \S(0,x-y_a).
\end{align*}
By the coupling of $\S$ with the parabolic Airy$_2$ process $\cP_1$ from item (ii) of Definition~\ref{d.airy sheet}, we see that
$$\wdf(x) \stackrel{d}{=} \cP_1(x-y_b) - \cP_1(x-y_a).$$
That the righthand side is a continuous random variable is an immediate consequence of Corollary~\ref{c.abs cont} on the absolute continuity of increments of $\cP_1$ to Brownian motion.

For the case of $\rho_c^h$, it is enough to show that 
$\P(\wdf(\rho_c^h) = \wdf(\rho_c^h + \varepsilon)) = 0$
for each rational $\varepsilon>0$. Now, since $\rho_c^h < \infty$ almost surely (by Lemma~\ref{l.strong non-degeneracy of E}) and $\wdf(\rho_c^h) = \wdf(c) + h$, it is further sufficient to show that, for each rational $\varepsilon>0$,
$$\P\left(\wdf(\rho_c^h+\varepsilon) = \wdf(c) + h, \rho_c^h < c + \varepsilon^{-1}\right) = 0.$$
The event in this equation is a function of $\wdf(\cdot) - \wdf(c)$ on $[c, c+\varepsilon^{-1}]$. By Theorem~\ref{t.patchwork}, this process is absolutely continuous to $\smash{\cL(\cdot) - \cL(1)|_{[1,1+\varepsilon^{-1}]}}$. Since the event in the previous display has probability zero for the local time increment process, Theorem~\ref{t.patchwork} implies that the previous display holds and so $\rho_c^h$ is almost surely a point of increase of $\wdf$.

Finally, we consider $\xi_{[c,d]}$. By definition, $\xi_{[c,d]}$ is a non-constant point. Now observe that there are only countably many non-constant points of $\wdf$ that are not points of increase, again by considering a small interval to the right of any such point and invoking the denseness of rationals. Since the measure with distribution function $\wdf(\cdot)-\wdf(c)/(\wdf(d)-\wdf(c))$ has no atoms, $\xi_{[c,d]}$ almost surely does not lies in this countable subset of $\NC(\wdf)$. This completes the proof of Lemma~\ref{l.random times are points of increase}.
\end{proof}
}

It remains to prove Proposition~\ref{p.Z_1 Z_3 never coincide}, that $Z^{\lambda,b}_1$ is always strictly greater than $Z^{\lambda, b}_3$, which is the task of the next section. The proof will also make use of the earlier and yet-to-be-proved Proposition~\ref{p.maximizer index is finite for general i}, on the finiteness of the maximizing index in the definition of $Z^{\lambda,b}_i$, which we will prove in Section~\ref{s.Z^lambda index finite}.

\subsection{Proving Proposition~\ref{p.Z_1 Z_3 never coincide}}\label{s.Z_1 Z_3 never collide}

As mentioned, the proof relies on the fact that two-dimensional Brownian motion almost surely never hits any given point in the plane, which we recall precisely.

\begin{lemma}[Corollary 2.26 of \cite{morters2010brownian}]\label{l.2d BM zero}
Let $x, y\in\R^2$ and $B:[0,\infty)\to\R^2$ be two-dimensional Brownian motion begun at $x$. Almost surely, $\P(y\in\{B(t): t>0\}) = 0$.
\end{lemma}

The proof of Proposition~\ref{p.Z_1 Z_3 never coincide} will convert the condition that $\smash{Z^{\lambda, b}_1}$ and $\smash{Z^{\lambda, b}_3}$ meet into one on $\smash{Z^{\lambda, b}_3}$, $\cP_1$, and $\cP_2$. More precisely, we will use the recursive representation of $Z^{\lambda, b}_i$ stated in Lemma~\ref{l.Z_i^lambda recursion} to show that any point where $\smash{Z^{\lambda, b}_1}$ and $\smash{Z^{\lambda, b}_2}$ agree is a running maximizer of $\smash{Z^{\lambda, b}_2 - \cP^\lambda_1}$, and a similar statement for points of agreement of $\smash{Z^{\lambda, b}_2}$ and $\smash{Z^{\lambda, b}_3}$. Thus we introduce the following notation. For a stochastic process $X:[0,\infty)\to \R$ and $\varepsilon>0$, let $\rec_\varepsilon(X)\subseteq [\varepsilon,\infty)$ be the set of running maximizers of $X|_{[\varepsilon, \infty)}$, i.e., the points $x\geq \varepsilon$ such that $X(x) = \max_{\varepsilon\leq s\leq x} X(s)$.

\begin{proof}[Proof of Proposition~\ref{p.Z_1 Z_3 never coincide}]
It is enough to show that for every $0<\varepsilon<T$ rationals, there almost surely does not exist $\tau\in[\varepsilon,T]$ such that $Z^{\lambda, b}_1(\tau) = Z^{\lambda, b}_3(\tau)$, since $Z^{\lambda, b}_1(x)\geq Z^{\lambda, b}_3(x)$ for all $x\geq 0$. Fix such $\varepsilon$ and $T$ now.

Suppose to the contrary that $\tau\in[\varepsilon,T]$ is such that $Z^{\lambda, b}_{1}(\tau) = Z^{\lambda, b}_{2}(\tau) = Z^{\lambda, b}_{3}(\tau)$, since always $Z_1^\lambda(x) \geq Z_2^\lambda(x) \geq Z_3^\lambda(x)$. By Lemma~\ref{l.Z_i^lambda recursion} and the definition \eqref{e.G' defn} of $\PT$, $Z_2^\lambda(\tau) = Z_3^\lambda(\tau)$ implies that
\begin{align*}
Z^{\lambda, b}_3(\tau) - \cP^\lambda_2(\tau) &= b^\lambda_2\vee \max_{0\leq s\leq \tau} \left(Z^{\lambda, b}_3(s) - \cP^\lambda_2(s)\right)
\geq \max_{\varepsilon/2\leq s\leq \tau} \left(Z^{\lambda, b}_3(s) - \cP^\lambda_2(s)\right).
\end{align*}
Similarly, $Z^{\lambda, b}_1(\tau) = Z^{\lambda, b}_3(\tau)$ implies that
\begin{align*}
Z^{\lambda, b}_3(\tau) - \cP_1^\lambda(\tau) = b^\lambda_1\vee \max_{0\leq s\leq \tau} \left(Z^{\lambda, b}_2(s) - \cP^\lambda_1(s)\right)
\geq \max_{\varepsilon/2\leq s\leq \tau} \left(Z^{\lambda, b}_3(s) - \cP^\lambda_1(s)\right),
\end{align*}
the inequality using also that $\smash{Z^{\lambda, b}_2(x)\geq Z^{\lambda, b}_3(x)}$ for all $x$. Thus, we see that $\smash{Z^{\lambda, b}_1(\tau)}= \smash{Z^{\lambda, b}_2(\tau)} = \smash{Z^{\lambda, b}_3(\tau)}$ implies that
\begin{equation}\label{e.tau membership}
\tau \in \rec_{\varepsilon/2}(Z^{\lambda, b}_3 - \cP^\lambda_1) \cap \rec_{\varepsilon/2}(Z^{\lambda, b}_3 - \cP^\lambda_2)\cap[\varepsilon,T].
\end{equation}
We claim that the set on the righthand side is almost surely empty.

First, recall from Proposition~\ref{p.Br abs cont of Z} that $(\cP^\lambda_1, \cP^\lambda_2, Z^{\lambda, b}_3)|_{[\varepsilon/2,T]} \ll (B_1,B_2,B_3)_{[\varepsilon/2,T]}$, where the $B_i$ are independent rate two Brownian motions. 
%
%
That $(\smash{\cP^\lambda_1, \cP^\lambda_2, Z^{\lambda, b}_3})|_{[\varepsilon/2,T]} \ll (B_1,B_2,B_3)_{[\varepsilon/2,T]}$ implies that $(\smash{Z^{\lambda, b}_3 - \cP^\lambda_1, Z^{\lambda, b}_3 - \cP_2^\lambda})|_{[\varepsilon/2,T]}$ is absolutely continuous to the restriction to $[\varepsilon/2,T]$ of a pair of (non-trivially) correlated Brownian motions, which is in turn absolutely continuous to $(B,B')|_{[\varepsilon/2,T]}$, where $B$ and $B'$ are a pair of independent Brownian motions on $[0,\infty)$.

Since the set on the righthand side of \eqref{e.tau membership} is a function of $(Z^{\lambda, b}_3 - \cP_1^\lambda)|_{[\varepsilon/2,T]}$ and $(Z^{\lambda, b}_3 - \cP_2^\lambda)|_{[\varepsilon/2,T]}$, to prove that that set is empty almost surely, it is sufficient to show that, almost surely,
$$\rec_{\varepsilon/2}(B)\cap\rec_{\varepsilon/2}(B')\cap[\varepsilon,T] = \emptyset.$$

Now by the Markov property of Brownian motion, and since $\rec_{\varepsilon/2}(X)$ is unaffected by vertical shifts to $X$, $\rec_{\varepsilon/2}(B)\cap \rec_{\varepsilon/2}(B')$ has the same distribution as $\varepsilon/2 + \rec_0(B)\cap \rec_0(B')$, where $x+A$ for $x\in\R$ and a set $A$ is the set $\{x+y: y\in A\}$.

By L\'evy's identity (Proposition~\ref{p.levy identity}), $(\varepsilon/2+\rec_0(B)\cap\rec_0(B'))\cap[\varepsilon, \infty) = \emptyset$ almost surely is equivalent to two independent Brownian motions almost surely not sharing a zero after time $\varepsilon/2$. This is because, if $M$ is the running maximum of $B$, then $\rec_0(B)$ is the set of points where $M = B$, i.e., $M-B=0$; by L\'evy's identity, these have the distribution of the set of points where $|B|=0$. That independent Brownian motions almost surely do not share a zero after time $\varepsilon/2$ is implied by Lemma~\ref{l.2d BM zero}, thus completing the proof of Proposition~\ref{p.Z_1 Z_3 never coincide}.
\end{proof}

\subsection{Finite maximizing index for \texorpdfstring{$Z^{\lambda,b}_i$}{Z\string^\{lambda,b\}\_i}}
\label{s.Z^lambda index finite}


Proposition~\ref{p.maximizer index is finite for general i} is obtained by a somewhat delicate geometric argument, using the $i=1$ case proved in \cite{sarkar2020brownian} and recorded in this article as Lemma~\ref{l.Airy sheet to boundary LPP problem}.
{While the proof strategy as adopted in \cite{sarkar2020brownian} involving infinite geodesics and the fact that they follow roughly parabolic trajectories provides an approach to proving Proposition~\ref{p.maximizer index is finite for general i},  here we choose to adopt a different strategy, in keeping with our reliance on arguments involving boundary data instead of infinite geodesics.}

In the coming pages, we will refer several times to the geodesic implicit in the definition \eqref{e.Z_i^lambda defn} of $\smash{Z^{\lambda,b}_1(z)}$. This means, if the minimum index at which the supremum in that definition is attained is $i_0$ (which is known to be finite by Lemma~\ref{l.Airy sheet to boundary LPP problem}), to consider the geodesic from $(\lambda, i_0)$ to $(\lambda+z,1)$ in the environment defined by $\cP$. {In the case that this geodesic is not unique, we will consider the left-most one; it is a standard consequence of planarity that this is a well-defined notion, and we refer the reader, for example, to \cite[Lemma~3.5]{dauvergne2018directed}.}

The basic idea of Proposition~\ref{p.maximizer index is finite for general i} is geometric and uses the ordering or monotonicity properties of geodesics. Namely, consider the geodesic $\pi^z$ implicitly defined by $Z^{\lambda,b}_1(z)$. If for some $z\geq x$ this geodesic leaves line $j$ after $\lambda+x$, then the maximizing index for $\smash{Z^{\lambda,b}_j(x)}$ must be at most the starting line number of $\pi^z$, by geodesic ordering, which we know is finite from Lemma~\ref{l.Airy sheet to boundary LPP problem} (Lemma~\ref{l.upper bound on i^th line starting index}). See Figure~\ref{f.Z_i maximizer argument}. The proof of Proposition~\ref{p.maximizer index is finite for general i} comes down to showing that the location at which $\pi^z$ leaves line $j$ goes to $\infty$ as $z\to\infty$, thus covering all values of $x$ for $\smash{Z^{\lambda,b}_j(x)}$. This is shown by arguing that, otherwise, the geodesic from vertical line $\lambda+x$ (with boundary data $\smash{\{b_i^{\lambda+x}\}_{i\in\N}}$) to $(\lambda+z,1)$ would have uniformly bounded starting line number for all $z$ (Lemma~\ref{l.upper bound on i^lambda in bad case}); that this is not possible is Lemma~\ref{l.i limit is infinite}.

\begin{figure}
\includegraphics{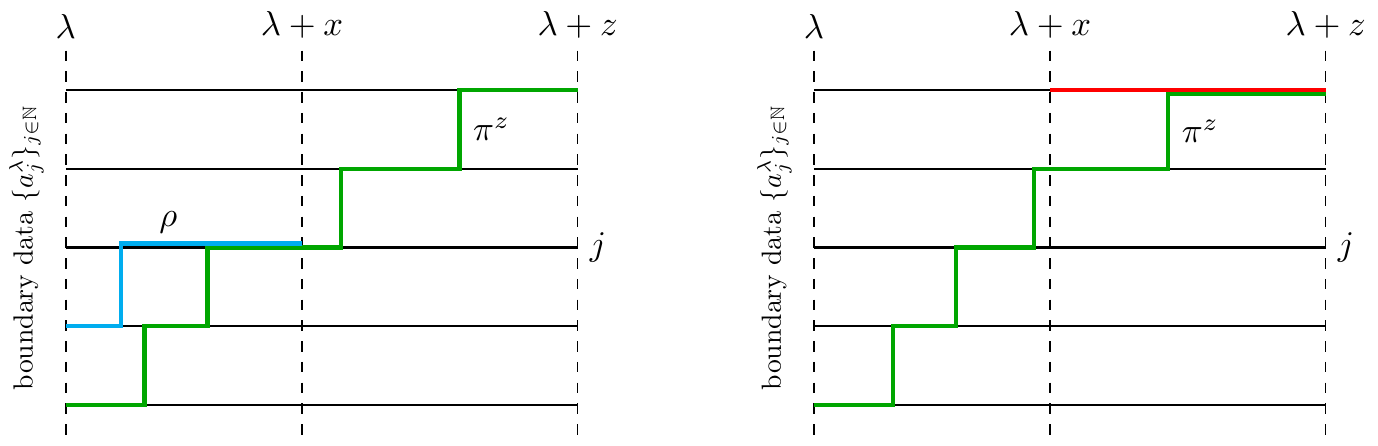}
\caption{On the left panel is depicted the situation where the geodesic $\pi^z$ (in green) from the vertical line $\lambda$ with boundary data $\{b^\lambda_i\}_{i\in\N}$ to $(\lambda+z,1)$ leaves line $j$ after $\lambda+x$. In this case, planarity and the weight maximization property of geodesics implies that the geodesic $\rho$ (in light blue) from the vertical line $\lambda$, again with boundary data $\{b^\lambda_i\}_{i\in\N}$, to $(\lambda+x,j)$ is to the left of $\pi^z$; in particular, $\rho$ starts at a line whose index is lower than that of $\pi^z$. (In fact in the depiction, since $\pi^z$ is at line $j$ at $\lambda+x$, the portion of $\pi^z$ till $\lambda+x$ must coincide with $\rho$. We have not shown this as it need not hold in general when $\pi^z$ is at an index greater than $j$ at $\lambda+x$.) In the right panel, $\pi^z$ leaves line $j$ before $\lambda+x$; this implies that the red geodesic from $\lambda+x$, with boundary data $\{b^{\lambda+x}_i\}_{i\in\N}$, to $(\lambda+z,1)$ starts at a line above the $j$\textsuperscript{th} one.}\label{f.Z_i maximizer argument}
\end{figure}

\begin{lemma}\label{l.upper bound on i^th line starting index}
Fix $\lambda\in \R$ and $0\leq x\leq z$, and let $\pi^z$ be the geodesic implicit in the definition \eqref{e.Z_i^lambda defn} of $\smash{Z^{\lambda,b}_1(z)}$. Fix $j\in\N$ and suppose that $\pi^z$ exits line $j$ after $\lambda+x$. Then the maximizing index (the minimum such in the case of non-uniqueness) in the definition of $\smash{Z^{\lambda,b}_j(x)}$ is at most $i^\lambda(z)$.
\end{lemma}

\begin{proof}
Let $\ell$ be the minimum index achieving the supremum in the definition \eqref{e.Z_i^lambda defn} of $Z^{\lambda,b}_j(x)$. Since $z$ is fixed, let $\pi=\pi^z$. Suppose to the contrary that $\pi$ leaves line $j$ after $\lambda+x$, and $\smash{\ell > i^\lambda(z)}$. 

Consider the geodesic from $(\lambda, \ell)$ to $(\lambda+x, j)$ and call it $\rho$. Note that our hypothesis that $\ell> i^\lambda(z)$ and planarity imply that $\rho$ and $\pi$ must have non-empty intersection, though it may possibly be just $(\lambda+x, j)$. Let $y\in \R\times\N$ be the point of intersection with minimum first coordinate. Consider the restriction of $\pi$ from its starting point to $y$, which we label $\pi'$, and we define $\rho'$ from $\rho$ similarly. Recall from Definition~\ref{d.lpp} that the weight of a path $\gamma$ in the environment given by $\cP$ is denoted by $\cP[\gamma]$. Note that
\begin{equation}\label{e.pi' rho' inequality}
\cP[\pi'] + b^\lambda_{i^\lambda(z)} \geq \cP[\rho'] + b_\ell^\lambda,
\end{equation}
for otherwise we could replace $\pi'$, as portion of the path $\pi$, by $\rho'$ and obtain a path from the vertical line $\lambda$ to $(\lambda+z,1)$ of greater weight (taking into account the boundary data at $\lambda$).

However, \eqref{e.pi' rho' inequality} contradicts the definition of $k$ as the minimum index achieving a certain maximum. This is because we can consider the path $\rho''$ obtained by replacing $\rho'$, as a portion of the path $\rho$, by $\pi'$; this path goes from $(\lambda, i^\lambda(z))$ to $(\lambda+x, j)$ and, by \eqref{e.pi' rho' inequality}, satisfies
$$\cP[\rho''] + b^\lambda_{i^\lambda(z)} \geq \cP[\rho] + b^\lambda_\ell;$$
but then $\ell$ is not the minimum index achieving the supremum in the definition of $Z^{\lambda, b}_j(x)$ since $\ell > i^\lambda(z)$, whence the contradiction. Thus the proof of Lemma~\ref{l.upper bound on i^th line starting index} is complete.
\end{proof}

\begin{lemma}\label{l.upper bound on i^lambda in bad case}
Fix $\lambda\in\R$ and $0\leq x\leq z$, and let $\pi^z$ be the geodesic implicit in the definition \eqref{e.Z_i^lambda defn} of $Z^{\lambda, b}_1(z)$. Let $j$ be the index of the line that $\pi^z$ is on at $\lambda+x$. Then $j\geq i^{\lambda+x}(z-x)$.
\end{lemma}

\begin{proof}
For ease of notation, let $\pi^z = \pi$, $i^{\lambda+x} = i^{\lambda+x}(z-x)$, and $i^\lambda = i^\lambda(z)$. Recall from Lemma~\ref{l.Airy sheet to boundary LPP problem} that $i^{\lambda+x}$ and $i^\lambda$ are respectively the minimum indices which achieve the supremums in
\begin{equation}\label{e.i^lambda defn}
\begin{split}
\S(y_b, \lambda+z) &= \sup_{i\in\N}\left\{b^{\lambda+x}_i + \cP[(\lambda+x,i) \to (\lambda+z,1)]\right\}\\
&= \sup_{i\in\N}\left\{ b^\lambda_i + \cP[(\lambda, i)\to (\lambda+z,1)]\right\}.\end{split}
\end{equation}
We \emph{claim} that
\begin{equation}\label{e.claim on shifted a relationship}
b^{\lambda+x}_j + \cP[(\lambda+x,j)\to (\lambda+z,1)] \geq b^\lambda_{i^\lambda} + \cP[(\lambda, i^\lambda)\to (\lambda+z,1)];
\end{equation}
though we will not need this, we note that combining this inequality with \eqref{e.i^lambda defn} implies that we actually have equality.

We first prove Lemma~\ref{l.upper bound on i^lambda in bad case} given this claim. By the definition of $i^\lambda$ and since both supremums in \eqref{e.i^lambda defn} equal $\S(y_b,\lambda+z)$, \eqref{e.claim on shifted a relationship} implies that the first supremum in \eqref{e.i^lambda defn} is achieved at $i=j$. Then by the minimality in the definition of $i^{\lambda+x}$, it follows that $i^{\lambda+x} \leq j$, which is the claim of Lemma~\ref{l.upper bound on i^lambda in bad case}.

It remains to prove \eqref{e.claim on shifted a relationship}. Since $j$ is the index of the line $\pi$ (the geodesic from $(\lambda, i^\lambda)$ to $(\lambda+z,1)$) is on at $\lambda+x$, we see that 
$$\cP[(\lambda, i^\lambda)\to (\lambda+z,1)] = \cP[(\lambda, i^\lambda) \to (\lambda+x,j)] + \cP[(\lambda+x,j) \to (\lambda+z,1)].$$
Thus verifying \eqref{e.claim on shifted a relationship} reduces to showing
$$b_j^{\lambda+x} \geq b^\lambda_{i^\lambda} + \cP[(\lambda,i^\lambda)\to (\lambda+x, j)].$$
To show this, observe that, by the definition \eqref{e.a b definition} of $b^\lambda_i$ for $i\in\N$ and since $\cP[x\to z] \geq \cP[x\to y] + \cP[y\to z]$ for any coordinates $x,y,z\in\R\times\N$,
\begin{align*}
\MoveEqLeft
b^\lambda_{i^\lambda} + \cP[(\lambda, i^\lambda)\to(\lambda+x,j)]\\
&= \lim_{k\to\infty} \left(\cP[(y_b)_k\to (\lambda, i^\lambda)] + \cP[(\lambda, i^\lambda)\to(\lambda+x,j)] - \cP[(y_b)_k \to (\lambda, 1)] + \S(y_b, \lambda)\right)\\
&\leq \lim_{k\to\infty} \Bigl(\cP[(y_b)_k\to (\lambda+x,j)] - \cP[(y_b)_k \to (\lambda, 1)] + \S(y_b, \lambda)\Bigr)\\
&= \lim_{k\to\infty} \Bigl(\cP[(y_b)_k\to (\lambda+x,j)] - \cP[(y_b)_k\to(\lambda+x,1)] + \S(y_b, \lambda+x)\Bigr)\\
&\qquad+ \lim_{k\to\infty}\Bigl(\cP[(y_b)_k\to(\lambda+x,1)] - \cP[(y_b)_k \to (\lambda, 1)] + \S(y_b, \lambda) - \S(y_b,\lambda+x)\Bigr).
\end{align*}
Now, the first limit is $b^{\lambda+x}_{j}$ by definition, while the second is zero by item (ii) of Definition~\ref{d.airy sheet}. This completes the verification of \eqref{e.claim on shifted a relationship} and thus the proof of Lemma~\ref{l.upper bound on i^lambda in bad case}.
\end{proof}

The next lemma says that the geodesic from the vertical line at $\lambda$, with boundary data $\{b^\lambda_i\}_{i\in\N}$, to $(\lambda+z,1)$ must have starting line index go to $\infty$ as $z\to\infty$. Recall the general notation introduced in Section~\ref{s.geometry}: for $y>0$, $\lambda\in\R$, and $z\geq \lambda$, $i^\lambda_y(z)$ is be the starting line index of the geodesic from vertical line at $\lambda$ to $(\lambda+z,1)$, where the boundary data at $\lambda$ is given by $\smash{\{b^{\lambda,y}_i\}_{i\in\N}}$. Lemma~\ref{l.i j properties} says that, if $y < y'$, then $i^\lambda_y(z)\leq i^\lambda_{y'}(z)$ for all $z\geq 0$, and we will make use of this.

\begin{lemma}\label{l.i limit is infinite}
Fix $y> 0$ and $\lambda\in\R$. Then, almost surely, $\lim_{z\to\infty} i^\lambda_{y}(z) = \infty$.
\end{lemma}

\begin{proof}
Suppose not. Then with positive probability, there exists $K$ such that
$$i^\lambda_{y}(z) \leq K$$
for all $z\geq 0$. By Lemma~\ref{l.i j properties}, we have that $i^\lambda_{y'}(z) \leq i^\lambda_{y}(z)$ for all $z\geq 0$ and $0< y' < y$. Thus, on the same event, by the pigeonhole principle, there exist positive rationals $y_1 < y_2$ such that, for all large~$z$,
$i^\lambda_{y_1}(z) = i^\lambda_{y_2}(z).$

But then, by Lemma~\ref{l.Airy sheet to boundary LPP problem}, we see that for all such large $z$,
$$\S(y_2, \lambda+z) - \S(y_1, \lambda+z) = c,$$
where $c$ is a (random) finite constant. This contradicts Lemma~\ref{l.strong non-degeneracy of E}, which, by a union bound over positive rational starting points, implies that $\lim_{z\to\infty} (\S(y_2,\lambda+ z) - \S(y_1, \lambda+z)) = \infty$ almost surely.
\end{proof}

We may now give the proof of Proposition~\ref{p.maximizer index is finite for general i}. It may be useful to refer to Figure~\ref{f.Z_i maximizer argument} once again.

\begin{proof}[Proof of Proposition~\ref{p.maximizer index is finite for general i}]
The case of $j=1$ is asserted by Lemma~\ref{l.Airy sheet to boundary LPP problem} (which is being cited from \cite{sarkar2020brownian}), and we will rely on this input to prove the general case.

Let us denote the smallest index achieving the maximum for $Z_j^{\lambda, b}(x)$ by $\ell^\lambda_j(x)$, which is infinite in the case that the maximum is not achieved at any finite index. By a planarity argument as in Lemma~\ref{l.i j properties} (which asserts the following for $j=1$), we see that $\ell^\lambda_j(x)$ is non-decreasing in $x$. Thus it is sufficient to prove that, almost surely, $\smash{\ell^\lambda_j(x)}$ is finite for each $x\in\N$. Fix an $x\in\N$ now.

By Lemma~\ref{l.upper bound on i^th line starting index}, we see that $\ell^\lambda_j(x) \leq i^\lambda(z)$ whenever $z$ is such that the geodesic $\pi^z$ implicit in the definition \eqref{e.Z_i^lambda defn} of $\smash{Z^{\lambda, b}_1(z)}$ exits line $j$ after $\lambda+x$. Since $\smash{i^\lambda(z) < \infty}$ for all $z$ by the $j=1$ case, we are done if we can show that, almost surely, for all large enough $z$, $\pi^z$ exits line $j$ after $\lambda+x$.

Suppose not. Then on a positive probability event the index of the line that $\pi^z$ is on at $\lambda+x$ is at most $j-1$ for all $z\geq x$. Then on the same event, by Lemma~\ref{l.upper bound on i^lambda in bad case}, for all $z\geq x$, $i^{\lambda+x}(z-x) \leq j-1$. This contradicts Lemma~\ref{l.i limit is infinite}, which asserts that, with probability one, $i^{\lambda+x}(z-x)\to \infty$ as $z\to\infty$. This completes the proof of Proposition~\ref{p.maximizer index is finite for general i}.
%
\end{proof}

\bibliographystyle{alpha}
\bibliography{weight-diff-local-time}

\begin{thebibliography}{CHHM21}

\bibitem[AKQ14]{alberts2014continuum}
Tom Alberts, Konstantin Khanin, and Jeremy Quastel.
\newblock The continuum directed random polymer.
\newblock {\em Journal of Statistical Physics}, 154(1):305--326, 2014.

\bibitem[BBS20]{balazs2020non}
M{\'a}rton Bal{\'a}zs, Ofer Busani, and Timo Sepp{\"a}l{\"a}inen.
\newblock Non-existence of bi-infinite geodesics in the exponential corner
  growth model.
\newblock In {\em Forum of Mathematics, Sigma}, volume~8. Cambridge University
  Press, 2020.

\bibitem[Bef08]{beffara2008dimension}
Vincent Beffara.
\newblock The dimension of the {SLE} curves.
\newblock {\em The Annals of Probability}, 36(4):1421--1452, 2008.

\bibitem[BGH19a]{basu2019fractal}
Riddhipratim Basu, Shirshendu Ganguly, and Alan Hammond.
\newblock Fractal geometry of {A}iry$_2$ processes coupled via the {A}iry
  sheet.
\newblock {\em arXiv preprint arXiv:1904.01717}, 2019.

\bibitem[BGH19b]{bates2019hausdorff}
Erik Bates, Shirshendu Ganguly, and Alan Hammond.
\newblock Hausdorff dimensions for shared endpoints of disjoint geodesics in
  the directed landscape.
\newblock {\em arXiv preprint arXiv:1912.04164}, 2019.

\bibitem[BGHK21]{basu2019lower}
Riddhipratim Basu, Shirshendu Ganguly, Milind Hegde, and Manjunath Krishnapur.
\newblock Lower deviations in {$\beta$}-ensembles and law of iterated logarithm
  in last passage percolation.
\newblock {\em Israel Journal of Mathematics}, 2021+.
\newblock To appear.

\bibitem[CH14]{corwin2014brownian}
Ivan Corwin and Alan Hammond.
\newblock Brownian {G}ibbs property for {A}iry line ensembles.
\newblock {\em Inventiones mathematicae}, 195(2):441--508, 2014.

\bibitem[CHH19]{calvert2019brownian}
Jacob Calvert, Alan Hammond, and Milind Hegde.
\newblock Brownian structure in the {KPZ} fixed point.
\newblock {\em arXiv preprint arXiv:1912.00992}, 2019.

\bibitem[CHHM21]{KPZfixedptHD}
Ivan Corwin, Alan Hammond, Milind Hegde, and Konstantin Matetski.
\newblock Exceptional times when the {KPZ} fixed point violates {J}ohansson's
  conjecture on maximizer uniqueness.
\newblock {\em arXiv preprint arXiv:2101.04205}, 2021.

\bibitem[CLP19]{cator2019attractiveness}
Eric~A Cator, Sergio~I L{\'o}pez, and Leandro~PR Pimentel.
\newblock Attractiveness of {B}rownian queues in tandem.
\newblock {\em Queueing Systems}, 92(1-2):25--45, 2019.

\bibitem[CP15]{cator2015local}
Eric Cator and Leandro~PR Pimentel.
\newblock On the local fluctuations of last-passage percolation models.
\newblock {\em Stochastic Processes and their Applications}, 125(2):538--551,
  2015.

\bibitem[CQR15]{corwin2015renormalization}
Ivan Corwin, Jeremy Quastel, and Daniel Remenik.
\newblock Renormalization fixed point of the {KPZ} universality class.
\newblock {\em Journal of Statistical Physics}, 160(4):815--834, 2015.

\bibitem[Dau21]{dauvergne2021isometries}
Duncan Dauvergne.
\newblock Last passage isometries for the directed landscape.
\newblock {\em arXiv preprint arXiv:2106.07566}, 2021.

\bibitem[DG21]{das2021law}
Sayan Das and Promit Ghosal.
\newblock Law of iterated logarithms and fractal properties of the {KPZ}
  equation.
\newblock {\em arXiv preprint arXiv:2101.00730}, 2021.

\bibitem[DOV18]{dauvergne2018directed}
Duncan Dauvergne, Janosch Ortmann, and B{\'a}lint Vir{\'a}g.
\newblock The directed landscape.
\newblock {\em arXiv preprint arXiv:1812.00309}, 2018.

\bibitem[DSV20]{dauvergne2020three}
Duncan Dauvergne, Sourav Sarkar, and B{\'a}lint Vir{\'a}g.
\newblock Three-halves variation of geodesics in the directed landscape.
\newblock {\em arXiv preprint arXiv:2010.12994}, 2020.

\bibitem[DV21]{dauvergne2018basic}
Duncan Dauvergne and B{\'a}lint Vir{\'a}g.
\newblock Bulk properties of the {A}iry line ensemble.
\newblock {\em Ann. Probab.}, 2021+.
\newblock To appear.

\bibitem[DZ21]{dauvergne2021disjoint}
Duncan Dauvergne and Lingfu Zhang.
\newblock Disjoint optimizers and the directed landscape.
\newblock {\em arXiv preprint arXiv:2102.00954}, 2021.

\bibitem[GPS10]{garban2010fourier}
Christophe Garban, G{\'a}bor Pete, and Oded Schramm.
\newblock The {F}ourier spectrum of critical percolation.
\newblock {\em Acta Mathematica}, 205(1):19--104, 2010.

\bibitem[H{\"a}g08]{hagg2008local}
Jonas H{\"a}gg.
\newblock Local {G}aussian fluctuations in the {A}iry and discrete {PNG}
  processes.
\newblock {\em The Annals of Probability}, 36(3):1059--1092, 2008.

\bibitem[Ham16]{hammond2016brownian}
Alan Hammond.
\newblock Brownian regularity for the {A}iry line ensemble, and multi-polymer
  watermelons in {B}rownian last passage percolation.
\newblock {\em arXiv preprint arXiv:1609.02971}, 2016.

\bibitem[Ham19a]{hammond2017modulus}
Alan Hammond.
\newblock Modulus of continuity of polymer weight profiles in {B}rownian last
  passage percolation.
\newblock {\em The Annals of Probability}, 47(6):3911--3962, 2019.

\bibitem[Ham19b]{hammond2017patchwork}
Alan Hammond.
\newblock A patchwork quilt sewn from {B}rownian fabric: regularity of polymer
  weight profiles in {B}rownian last passage percolation.
\newblock In {\em Forum of Mathematics, Pi}, volume~7. Cambridge University
  Press, 2019.

\bibitem[Ham20]{brownianLPPtransversal}
Alan Hammond.
\newblock Exponents governing the rarity of disjoint polymers in {B}rownian
  last passage percolation.
\newblock {\em Proceedings of the London Mathematical Society},
  120(3):370--433, 2020.

\bibitem[HKPV09]{manjunath}
John~Ben Hough, Manjunath Krishnapur, Yuval Peres, and B\'alint Vir\'ag.
\newblock {\em Zeros of {G}aussian analytic functions and determinantal point
  processes}, volume~51.
\newblock American Mathematical Soc., 2009.

\bibitem[HPS15]{hammond2015local}
Alan Hammond, G{\'a}bor Pete, and Oded Schramm.
\newblock Local time on the exceptional set of dynamical percolation and the
  incipient infinite cluster.
\newblock {\em Annals of Probability}, 43(6):2949--3005, 2015.

\bibitem[Led18]{ledoux2018law}
Michel Ledoux.
\newblock A law of the iterated logarithm for directed last passage
  percolation.
\newblock {\em Journal of Theoretical Probability}, 31(4):2366--2375, 2018.

\bibitem[Mat99]{mattila1999geometry}
Pertti Mattila.
\newblock {\em Geometry of sets and measures in Euclidean spaces: fractals and
  rectifiability}.
\newblock Number~44. Cambridge University Press, 1999.

\bibitem[MP10]{morters2010brownian}
Peter M{\"o}rters and Yuval Peres.
\newblock {\em Brownian motion}, volume~30.
\newblock Cambridge University Press, 2010.

\bibitem[MQR21]{matetski2016kpz}
Konstantin Matetski, Jeremy Quastel, and Daniel Remenik.
\newblock The {KPZ} fixed point.
\newblock {\em Acta Mathematica}, 2021+.
\newblock To appear.

\bibitem[PS02]{prahofer2002PNG}
Michael Pr\"{a}hofer and Herbert Spohn.
\newblock Scale invariance of the {PNG} droplet and the {A}iry process.
\newblock {\em J. Statist. Phys.}, 108(5-6):1071--1106, 2002.
\newblock Dedicated to David Ruelle and Yasha Sinai on the occasion of their
  65th birthdays.

\bibitem[QR13]{quastel2013local}
Jeremy Quastel and Daniel Remenik.
\newblock Local behavior and hitting probabilities of the {A}iry$_1$ process.
\newblock {\em Probability Theory and Related Fields}, 157(3-4):605--634, 2013.

\bibitem[RS05]{rohde2005basic}
Steffen Rohde and Oded Schramm.
\newblock Basic properties of {SLE}.
\newblock {\em Annals of Mathematics}, pages 883--924, 2005.

\bibitem[RY13]{revuz2013continuous}
Daniel Revuz and Marc Yor.
\newblock {\em Continuous martingales and {B}rownian motion}, volume 293.
\newblock Springer Science \& Business Media, 2013.

\bibitem[SS21]{sep2021busemann}
Timo Sepp\"{a}l\"{a}inen and Evan Sorensen.
\newblock Busemann process and semi-infinite geodesics in brownian last-passage
  percolation.
\newblock {\em arXiv preprint arXiv:2103.01172}, 2021.

\bibitem[SV21]{sarkar2020brownian}
Sourav Sarkar and B{\'a}lint Vir{\'a}g.
\newblock Brownian absolute continuity of the {KPZ} fixed point with arbitrary
  initial condition.
\newblock {\em Ann. Probab.}, 2021+.
\newblock To appear.

\bibitem[TW94]{tracy1994level}
Craig~A Tracy and Harold Widom.
\newblock Level-spacing distributions and the {A}iry kernel.
\newblock {\em Communications in Mathematical Physics}, 159(1):151--174, 1994.

\bibitem[War07]{warren2007dyson}
Jon Warren.
\newblock Dyson's {B}rownian motions, intertwining and interlacing.
\newblock {\em Electronic Journal of Probability}, 12:573--590, 2007.

\end{thebibliography}
\end{document}